\documentclass[oneside]{amsart}
\usepackage{enumerate,amssymb}
\usepackage{mathrsfs}
\usepackage{mathtools}
\usepackage{tikz}
\newtheorem{thm}{Theorem}[section]
\newtheorem{coro}[thm]{Corollary}
\newtheorem{prop}[thm]{Proposition}

\newtheorem{claim}{Claim}

\newtheorem{lem}[thm]{Lemma}
\theoremstyle{definition}

\newtheorem{rems}[thm]{Remarks}
\newtheorem{question}[thm]{Question}

\theoremstyle{definition}
\newtheorem{defn}[thm]{Definition}
\newtheorem{ex}[thm]{Example}
\newtheorem{exs}[thm]{Examples}
\newcommand{\Rset}{\mathbb{R}}
\newcommand{\Nset}{\omega}
\newcommand{\Pset}{\omega^\omega}

\newcommand{\Cset}{2^\omega}
\newcommand{\cset}{2^{<\omega}}
\newcommand{\pset}{\omega^{<\omega}}
\newcommand{\nset}{\Nset_+}
\newcommand{\abs}[1]{\lvert#1\rvert}

\newcommand{\seq}[1]{\langle#1\rangle}
\newcommand{\hm}{\mathscr H}
\newcommand{\uhm}{\overline{\mathscr H}}
\newcommand{\pack}{\mathscr P}
\newcommand{\micr}{\boldsymbol{\varsigma}}
\newcommand{\eps}{\varepsilon}
\newcommand{\del}{\delta}
\newcommand{\subs}{\subseteq}
\renewcommand{\leq}{\leqslant}
\renewcommand{\geq}{\geqslant}

\DeclareMathOperator{\hdim}{\dim_{\mathsf{H}}}
\DeclareMathOperator{\pdim}{\dim_{\mathsf{P}}}
\DeclareMathOperator{\lpdim}{\underline{dim}_{\mathsf{P}}}
\DeclareMathOperator{\lbdim}{\underline{dim}_{\mathsf{B}}}
\DeclareMathOperator{\ubdim}{\overline{dim}_{\mathsf{B}}}
\DeclareMathOperator{\diam}{diam}

\DeclareMathOperator{\suc}{\mathsf{succ}}
\DeclareMathOperator{\ldens}{\underline{\mathsf{dens}}}
\DeclareMathOperator{\udens}{\overline{\mathsf{dens}}}
\newenvironment{enum}{\begin{enumerate}[\rm(i)]}{\end{enumerate}}
  {\begin{list}{$\blacktriangleright$}{\labelwidth1ex\setlength{\leftmargin}{1.5em}}}%
  {\end{list}}
  {\begin{list}{$\blacksquare$}{\labelwidth1ex\setlength{\leftmargin}{1.5em}}}%
  {\end{list}}
\newenvironment{itemyze}%
  {\begin{list}{\textbullet}{\labelwidth1ex\setlength{\leftmargin}{1.3em}}}%
  {\end{list}}
\newcommand{\si}{$\sigma$\nobreakdash-}

\newcommand{\AAA}{\mathscr{A}}
\newcommand{\II}{\mathscr{I}}
\newcommand{\JJ}{\mathscr{J}}

\newcommand{\GG}{\mathscr{G}}
\newcommand{\EE}{\mathscr{E}}
\newcommand{\MM}{\mathscr{M}}
\newcommand{\DD}{\mathscr{D}}
\newcommand{\DDs}{\clos{\mathscr{D}}}
\newcommand{\DDsh}{\mathscr{D}^\sharp}
\newcommand{\NN}{\mathscr{N}}

\newcommand{\clos}[1]{\overline{#1}}
\newcommand{\closs}{\clos s}
\newcommand{\el}[1]{\ell^{#1}}
\newcommand{\mult}[1]{{\mkern-2mu\times\mkern-2mu #1}}
\newcommand{\shift}[1]{{\mkern-2mu+\mkern-2mu #1}}
\newcommand{\concat}{\mkern-1mu^\smallfrown\mkern-4mu}

\newcommand{\wh}{\widehat}
\newcommand{\whs}{\widehat s}
\newcommand{\SEQ}{\mathbb{S}}
\newcommand{\Cube}{\mathbb{C}}
\newcommand{\cube}{\mathsf{C}}
\newcommand{\emany}{\exists^\infty}
\newcommand{\fmany}{\forall^\infty}
\newcommand{\rest}{{\restriction}}
\newcommand{\comp}{{\circ}}
\renewcommand{\concat}{^{\mkern-1mu\smallfrown}\mkern-3mu}

\newcommand{\klass}{\boldsymbol{\mathcal K}}

\newcommand{\mc}{\mathcal}

\newcommand{\harm}{\mathfrak{h}}
\newcommand{\geom}{\mathfrak{g}}
\newcommand{\cont}{\mathfrak{c}}

\newcommand{\para}{{\parallel}}
\newcommand{\parti}{\mathbb{P}}
\newcommand{\cyl}[1]{\langle#1\rangle}

\newcommand{\group}{\mathbb{G}}

\newcommand{\smz}{\ensuremath{\boldsymbol{\mathsf{Smz}}}}
\newcommand{\ssmz}{${\boldsymbol{\mathsf{Smz}}^\sharp}$}

\DeclareMathOperator{\non}{\mathsf{non}}
\DeclareMathOperator{\add}{\mathsf{add}}
\DeclareMathOperator{\cov}{\mathsf{cov}}
\DeclareMathOperator{\covs}{\mathsf{cov*}\mkern-3mu}
\DeclareMathOperator{\cof}{\mathsf{cof}}

\DeclareMathOperator{\micro}{\mathsf{micro}}

\newenvironment{cproof}{\noindent$\boldsymbol\vdash$}{\hfill\qed\par}

\begin{document}
\title
[More on dominated and microscopic sets]
{More on dominated and microscopic sets}

\author{Ond\v rej Zindulka}
\address
{Ond\v rej Zindulka\\
Department of Mathematics\\
Faculty of Civil Engineering\\
Czech Technical University\\
Th\'akurova 7\\
160 00 Prague 6\\
Czech Republic}
\email{ondrej.zindulka@cvut.cz}
\urladdr{http://mat.fsv.cvut.cz/zindulka}
\author{Piotr Nowakowski}
\address{Piotr Nowakowski
\\Faculty of Mathematics and Computer Science
\\University of Lodz
\\Banacha 22,
90-238 \L\'{o}d\'{z}
\\Poland\\
 ORCID: 0000-0002-3655-4991}
\email{piotr.nowakowski@wmii.uni.lodz.pl}
\author{Cristina Villanueva-Segovia}
\address
{Cristina Villanueva-Segovia\\
Universidad Nacional Autónoma de México\\
Instituto de Matemáticas, Unidad Cuernavaca\\
Av. Universidad s/n. Col. Lomas de Chamilpa\\
62210, Cuernavaca, Morelos\\
México 
}
\thanks{
The first author gratefully acknowledges support received from a DGAPA-PAPIIT grant IN107526
and a SECIHTI grant CBF-2025-I-898.
The third named author acknowledges the support of the 
\textit{Secretaría de Ciencia, Humanidades, Tecnología e Innovación} (\textsc{secihti}) 
through the posdoctoral fellowship granted (Grant No. [I1200/111/2024]) .
}

\begin{abstract}
Let $S$ be a family of sequences decreasing to zero.
A set $E$ in a metric space is \emph{$S$-dominated} if, for every $s\in S$, 
there exists a countable cov\-er $\{E_n\}$ of $E$ such that $\diam E_n<s_n$ for every $n$.
We continue the study of families of dominated sets initiated in \cite{micro1}. 
We look, e.g., into Cartesian products of dominated sets with strong measure
and interaction of microscopic and porous sets.
Based upon a tight relationship of Hausdorff measures and dominated sets, cardinal invariants 
of the ideals of dominated sets are determined and a related problem of Kwela~\cite{MR3482702}
is resolved.
\end{abstract}

\maketitle

\section{Introduction}
The subject of the present paper --- dominated sets --- arises from the notion of microscopic sets 
that was introduced by Appell in~\cite{MR1912017,MR2152488}, and Appell, D'Aniello and 
V\"ath~\cite{MR1909968,MR2290215}. 
By their definition,
a set $M\subs\Rset$ is \emph{microscopic} if for every $\eps>0$ there is
a cover $\seq{E_1,E_2,E_3,\dots}$ of $E$ such that $\diam E_n<\eps^n$ for all $n$.

Microscopic sets form a \si ideal contained in the ideal of Hausdorff dimension zero.
In particular, every microscopic set has Lebesgue measure zero. On the other hand,
each set of strong measure zero is microscopic, so microscopic sets provide an example of a nontrivial,
new \si ideal on $\Rset$.

Since the mentioned pioneering papers, many publications addressed microscopic sets; 
\cite{micro1} provides their list and also a brief account of the history of the notion.

Building upon ideas of Horbaczewska~\cite{MR3759529} and Karasi\'nska, Paszkiewicz and 
Wagner Bojakowska~\cite{MR3685162}, \cite{micro1} develops the theory of \emph{dominated sets} that extends 
that of microscopic sets: given a family $S$ of sequences of positive reals decreasing to zero,
a set $E\subs X$ in a metric space $X$ is called \emph{$S$-dominated}
if for each $s\in S$ there is a countable cover $\{E_n\}$ of $E$ such that $\diam E_n<s_n$ for all $n$. 
%
Basic properties of dominated sets, structural properties of the set $S$ that make the family of 
$S$-dominated sets an ideal, and the tight relationship of dominated sets and Hausdorff measures
are examined and established.

The present paper is a follow-up to~\cite{micro1}. 
It focuses on several topics: the behavior of Cartesian products of dominated sets,
algebraic sums of microscopic and porous sets, properties and impact of certain pathological microscopic 
and dominated sets first described by Kwela in~\cite{MR3482702}, and cardinal invariants of 
the ideals of $S$-dominated sets.

Section~\ref{sec:back} recalls the notions of dominated and microscopic sets  and the background material,
as established in \cite{micro1}. 

Section~\ref{sec:products} looks at Cartesian products. 
It is proven that the product of two dominated sets does not have to be dominated, 
however, the product of a dominated and sharply dominated set 
(a slightly stronger notion, cf.~Section~\ref{sec:sharp})
is dominated. We also show that a product of a dominated set and a strong measure 
zero set is consistently dominated and consistently fails to be dominated.
Similar results hold for the algebraic sums of (sharply) dominated 
sets in Polish groups equipped with invariant metrics. Perhaps the most unexpected result, 
Example~\ref{prodex}(ii), claims that there are dominated sets such that their product with \emph{any}
uncountable set is not dominated.

By the celebrated Galvin-Mycielski-Solovay Theorem~\cite{GMS}, a set $X\subs\Rset$ has strong measure 
zero if and only if the algebraic sum of $X$ with any meager set does not cover $\Rset$.
Strong measure zero sets are exactly sets dominated by all sequences, while microscopic sets are dominated
by geometric sequences. So maybe there is an analogue of Galvin-Mycielski-Solovay Theorem for 
microscopic sets. 
This is subject to Section~\ref{sec:porous}. Indeed, it turns out that replacing meager sets with porous sets
yields the following result: If $M\subs\Rset$ is microscopic and $P\subs\Rset$ is \si porous, then
$M+P\neq\Rset$. Since the proof exhibits that porous sets exactly and tightly couple
microscopic sets, one wonders if the converse implication holds. We show that it does not:
there are many counterexamples exhibiting the failure of all conceivable converses.

In~\cite{MR3482702}, Kwela proved that the additivity of the ideal of microscopic sets on $\Rset$ 
is $\omega_1$. His intricate proof involves a construction of a set with interesting and powerful properties
exceeding remarkably Kwela's goal. In Section~\ref{sec:kwela1} we extract and enhance the combinatorics of his
construction and extend it far beyond microscopic sets on $\Rset$.
In a little more detail, 
Kwela constructed a particular microscopic set $E\subs\Rset$ of reals
and proved that there are sets $S$ of cardinality $\omega_1$ such that $E+S$ is not microscopic,
thus proving that the additivity of microscopic sets is $\omega_1$.
As a starting point of Section~\ref{sec:kwela1}, we isolate the combinatorial condition that gives 
his set this feature and generalize it to dominated sets and general metric spaces, thus introducing, 
for each sequence $s$ decreasing to zero, the notion of $s$-Kwela set, cf.~Definition~\ref{def:infra}. 
Then we prove the existence of many $s$-Kwela sets that are mutually far apart
in appropriate Cantor sets.

In Section~\ref{sec:kwela2} we start with proving, for each $s$, the existence of $s$-Kwela sets
in Banach spaces $c_0$, $\el1$ and $\el2$.  
Then we present a wide class of the so called \emph{onion spaces} based upon a strong notion of porosity
that is due to Vallin~\cite{MR1228396}. This class includes (but is not limited to)
all Banach spaces, Euclidean spaces and complete spaces that are not totally disconnected.
We prove that every such space contains $s$-Kwela sets for every $s$ that decreases 
at least as fast as a geometric sequence.

One has to ask if the latter condition can be dropped. 
It turns out that the answer is negative:
somewhat unexpectedly, doubling metric spaces admit no $s$-Kwela sets for arbitrarily small sequences.
However, the gap between the condition guaranteeing $s$-Kwela set and the one forbidding it is
rather wide and there is a room for improvement.

The last part of Section~\ref{sec:kwela2} proves the property that makes Kwela sets interesting: 
If $\mc K$ is an uncountable family of $s$-Kwela sets that are mutually far apart, 
then their union is not $\closs$-dominated. A few consequences are derived: a Cartesian product
of an $s$-Kwela set and an uncountable set is not $\closs$-dominated and
in many spaces, if there is an $s$-Kwela set, then the ideal of $\closs$-dominated sets and 
the ideal of Hausdorff measure zero sets induced by any gauge are never equal.

As another consequence, if $\DD(\closs)$ denotes the ideal of $\closs$-dominated sets
if there are $s$-Kwela sets, then the additivity of $\DD(\closs)$ is $\omega_1$,
the cofinality of $\DD(\closs)$ is $\cont$ and there are $\cont$ many disjoint sets 
that are not $\closs$-dominated. This is proved in Section~\ref{sec:invar} along
estimates of the covering and uniformity numbers of $\DD(\closs)$ that are based upon 
corresponding estimates of Hausdorff null ideals proved in~\cite{invariants2} 
and the tight relationship of dominated sets and Hausdorff measures established in~\cite{micro1}.
We also solve a problem~\cite[4.7]{MR3482702} regarding additivity of a certain ideal.

\section{Background}\label{sec:back}
We denote by $\Nset$ and $\nset$ the sets of all and positive natural numbers, respectively.
The cardinality of a set $A$ is denoted by $\abs A$.
The letter $d$ serves as a generic notation for a metric 
and $\diam E$ denotes the diameter of a set $E$ in a  metric space. 
We use also $d(x,E)$ and $d(E,F)$ to denote the distance of a point $x$ from
the set $E$ and the lower distance of the sets $E,F$, respectively. The closed ball of 
radius $r$ centered at $x$ is denoted by $B(x,r)$ and
likewise $B(F,r)$ denotes the set $\{x:d(x,F)\leq r\}$.

\subsection*{Dominated sets}
We recall the notion of dominated sets and related notions. 
\begin{defn}[\cite{micro1}]\label{def:dom}
Let $\SEQ\subs(0,\infty)^{\nset}$ denote the set of all strictly decreasing sequences
of positive reals converging to $0$.
\begin{itemyze}
\item Say that $S\subs\SEQ$ is \emph{countably determined}
if there is a countable $S'\subs S$ such that $\forall s\in S\ \exists s'\in S'$ $s'\leq s$.
\item For $s\in\SEQ$ and $k\in\nset$ let $s^{\shift k}=\seq{s_{n+k}:n\in\nset}$ and 
$s^{\mult k}=\seq{s_{kn}:n\in\nset}$. Write $s^{\mult{}}$ instead of $s^{\mult2}$ and
$s^{\shift{}}$ instead of $s^{\shift1}$.
\item A set $S\subs\SEQ$ is \emph{multiplicatively complete} if
$\forall s\in S\ \exists s'\in S\ s'\leq s^{\mult{}}$. Note that if $S$ is multiplicatively complete, then
$\forall s\in S\ \forall k\ \exists s'\in S\ s'\leq s^{\mult k}$.
\item For $S\subs\SEQ$ we define the \emph{multiplicative completion of $S$} to be
$\clos S=\{s^{\mult k}:s\in S, k\in\nset\}$. 
It is clear that $\clos S$ is multiplicatively complete and $S\subs\clos S$.
For $s\in\SEQ$ we write $\closs$ instead of $\clos{\{s\}}$.
Note that $\closs=\{s^{\mult k}:k\in\nset\}$ and that $\closs$ is countably determined.
\item \emph{Additively complete} and \emph{additive completion} are defined likewise. 
Additive completion of $S$ is denoted by $\wh S$.
\end{itemyze}
\end{defn}
\begin{defn}
Let $X$ be a separable metric space.
\begin{itemyze}
\item Let $s\in\SEQ$. A sequence $\seq{E_n:n\in\nset}$ of sets in $X$ is \emph{$s$-fine} if
$\diam E_n<s_n$ for all $n\in\nset$.
\item Let $S\subs\SEQ$. A set $E\subs X$ is \emph{$S$-dominated} if for each $s\in S$ there is
an $s$-fine cover of $E$.
\item The family of all $S$-dominated sets in $X$ is denoted by $\DD_X(S)$ or,
if there is no danger of confusion, just $\DD(S)$.
\item The family of all sets in $X$ that are contained in a $\sigma$-compact $S$-dominated set 
is denoted by $\DDs_X(S)$ or just $\DDs(S)$.
\end{itemyze}
\end{defn}
\begin{defn}
We single out the microscopic sets, since we are using them as a template.
Let $\geom=\seq{2^{-n}:n\in\nset}$ denotes the binary geometric sequence.
A set is \emph{microscopic} if it is $\clos\geom$-dominated. 
\end{defn}
Given $S\subset\SEQ$, the families $\DD(S)$ and $\DDs(S)$ need not form an ideal. 
However, as shown in~\cite{micro1}, 
if $S$ is multiplicatively complete, then $\DD(S)$ and $\DDs(S)$ are \si ideals, and
if $S$ is additively complete, then $\DDs(S)$ is a \si ideal.

The following lemma will be convenient. Its proof is staightforward.
For each $k\in\Nset$ let 
\begin{equation}\label{parti}
  \parti_k=\{2^k(2j+1):j\in\Nset\}
\end{equation}
consisting of odd multiples of $2^k$. The family $\{\parti_k:k\in\Nset\}$
is a partition of $\nset$.
\begin{lem}\label{basisDomi2}
A set $E$ in a metric space is $\closs$-dominated if and only if there is an $s$-fine sequence 
$\seq{E_n:n\in\nset}$ such that $E\subs\bigcap_{k\in\Nset}\bigcup\{E_n:n\in\parti_k\}$.
\end{lem}

\subsection*{Hausdorff measures}
Dominated sets and sets of Hausdorff measure zero are profoundly related.
If $\del>0$, a cover $\mc A$ of a set $E\subs X$ is termed
\emph{$\del$-fine} if $\diam A<\del$ for all $A\in\mc A$.
Recall that if $\phi$ is a gauge,
the \emph{$\phi$-dimensional Hausdorff measure} $\hm^\phi(E)$ of
a set $E\subs X$ is defined as follows:
For each $\del>0$ set
\begin{alignat*}{2}
  &\hm^\phi_\delta(E)&&=
  \inf\left\{\sum\nolimits_n\phi(\diam E_n):
  \text{$\{E_n\}$ is a countable $\delta$-fine cover of $E$}\right\},\\
  &\hm^\phi(E)&&=\sup_{\delta>0}\hm^\phi_\delta(E).
\end{alignat*}
Properties of Hausdorff measures are well-known, see e.g. \cite{MR0281862}.
Hausdorff measures have a silent parameter: the underlying space $X$.
Sometimes (but rarely) we make it visible by writing $\hm^\phi_X$.
The \si ideal of sets of Hausdorff measure zero is written as 
$\NN(\hm^\phi)=\{E\subs X:\hm^\phi(E)=0\}$.

We refer to~\cite{micro1} for a detailed analysis of how the ideals $\DD(S)$ and
$\NN(\hm^\phi)$ are related. Here are two examples:
\begin{thm}[{\cite{micro1}}]\label{versus0}
Let $\phi$ be a gauge and $s\in\SEQ$. The following are equivalent.
\begin{enum}
\item $\phi\circ s\in\el1$,
\item for every metric space $X$, $\DD_X(\closs)\subs\NN(\hm_X^\phi)$,
\item $\DD_{\el2}(\closs)\subs\NN(\hm_{\el2}^\phi)$.
\end{enum}
\end{thm}
\begin{thm}[{\cite{micro1}}]\label{versus4}
If $\phi$ is a continuous gauge, then
$\NN(\hm^\phi)=\bigcup_{\phi\comp s\in\el1}\DD(\closs)$.
\end{thm}

\subsection*{Compact sets of Hausdorff measure zero}
The following modification of Hausdorff measure was introduced in~\cite{MR3114775,MR3946667}.
If $\phi$ is a gauge, for each $\del>0$ set define
\begin{alignat*}{2}
  &\uhm^\phi_\delta(E)&&=
  \inf\left\{\sum\nolimits_n\phi(\diam E_n):
  \text{$\{E_n\}$ is a finite $\delta$-fine cover of $E$}\right\},\\
  &\uhm_0^\phi(E)&&=\sup_{\delta>0}\uhm^\phi_\delta(E).
\end{alignat*}
Though not a measure, $\uhm_0^\phi$ is a metric outer measure.
Define $\EE(\hm^\phi)$ to be the \si ideal generated by sets $E$ such that $\uhm_0^\phi(E)=0$. 
The notation is explained by the following fact.
\begin{lem}
 If $X$ is complete, then 
\[
  \EE(\hm^\phi)=\{E\subs X:\exists \text{ \si compact } F\supseteq E, \, \hm^\phi(F)=0\}.
\]
\end{lem}
See~\cite{MR3114775,invariants2} for the proof and more information on these ideals.

\section{Sharply dominated sets}\label{sec:sharp}

We shall see later that dominated sets do not behave in cartesian products, while
a slight strengthening, the so called sharply dominated sets, do.

\begin{defn}[$\gamma$-groupable covers]  
Recall (cf.~\cite{MR3946667}): family $\{A_k:k\in\Nset\}$ is a \emph{$\gamma$-cover} of $X$ 
if every $x\in X$ is contained in all but finitely many $A_n$'s.
A family $\GG=\{G_n:n\in\Nset\}$ is called a \emph{$\gamma$-groupable} cover of $X$ if there is 
a disjoint collection $\{\mc F_k:k\in\Nset\}$ of finite subfamilies of $\GG$ 
such that each $x\in X$ is covered by all but finitely $\mc F_k$'s.
This can be also phrased as follows: $\{G_n:n\in\Nset\}$ is a $\gamma$-groupable cover of $X$ 
if there is a disjoint collection $\{F_k:k\in\Nset\}$ of finite subsets of $\Nset$ such that, 
letting $A_k=\bigcup_{n\in F_k}G_n$, the family $\{A_k:k\in\Nset\}$ is a $\gamma$-cover.

The families $\{\mc F_k:k\in\Nset\}$ or $\{F_k:k\in\Nset\}$ are referred to as witnessing families 
or witnessing sets. 
By \cite{invariants2}, the witnessing sets may be supposed to be consecutive intervals covering $\Nset$.
\end{defn}
\begin{defn}[sharp domination]
Let $S\subs\SEQ$. A set $E\subs X$ is \emph{sharply $S$-dominated} 
if for each $s\in S$ there is an $s$-fine $\gamma$-groupable cover of $E$.
The family of all sharply dominated subsets of $X$ is denoted by $\DDsh_X(S)$ or $\DDsh(S)$.
\end{defn}
\begin{prop}\label{sharp3}
Let $X$ be a complete space and $S\subs\SEQ$. 
\begin{enum}
\item $\DDsh(S)=\DDsh(\wh S)=\bigcap_{s\in S}\DDs(\whs)$.
\item $\DDsh(S)$ is a \si ideal.
\item If $S\subs\SEQ$ is additively complete, then $\DDs(S)\subs\DDsh(S)$.
\end{enum}
\end{prop}
\begin{proof}
(i)
Let $s\in\SEQ$. 
If $\seq{E_n:n\in\nset}$ is an $s$-fine $\gamma$-groupable cover
of $E$, then $\seq{E_n:n>k}$ is $s^{\shift k}$-fine for any $k\in\Nset$, 
and it is still a $\gamma$-groupable cover of $E$.

The inclusion $\DDsh(S)\subs\DDsh(\wh S)$ follows from this argument and the reverse inclusion is trivial. 

To prove $\DDsh(S)\subs\bigcap_{s\in S}\DDs(\whs)$ suppose $E$ is sharply $S$-dominated and $s\in S$. Let 
$\{E_n:n\in\nset\}$ be an $s$-fine $\gamma$-groupable cover of $E$ and $\{F_k:k\in\nset\}$ 
be a witnessing partition. We may also suppose that all $E_n$'s are closed.
For each $k,m$ let $A_k=\bigcup_{n\in F_k}E_n$ and $B_m=\bigcap_{k\geq m} A_k$.

For any $\eps>0$ there is $k$ such that $s_n<\eps$ for all $n\in F_k$. 
Since $F_k$'s are finite, it follows that each $B_m$ is totally bounded. 
Since each $A_k$ is a union of finitely many closed sets, each $B_m$ is also closed. 
Overall, each $B_m$ is compact.
Consider $B=\bigcup_{m\in\Nset}B_m$. It is \si compact, contains $E$ and $\{E_n:n\geq m\}$ 
is a cover of $B$ for each $m$, so $B$ is $\whs$-dominated. We proved that $E\in\DDs(\whs)$,
as required.

To prove the reverse inclusion it is enough to show that $\DDs(\whs)\subs\DDsh(s)$.
So let $E$ be \si compact $\whs$-dominated set and $\{K_n:n\in\nset\}$ 
an increasing cover of $E$ by compact $\whs$-dominated sets. 
Using compactness of $K_n$'s, find by recursion an increasing sequence $k_n$ 
and a finite open cover $\{U_i:i\in[k_{n},k_{n+1})\}$ of $K_n$ such that $\diam U_i<s_i$. 
It is clear that $\mc U=\{U_i:i\in\nset\}$ is an $s$-fine cover and that 
$\{U_i:i\in[k_{n},k_{n+1})\}$ are the finite families witnessing that $\mc U$ is 
$\gamma$-groupable.

(ii) We know from \cite[6.2]{micro1} that $\DDs(\whs)$ is a \si ideal, 
so (ii) follows immediately from (i).

(iii) Since $S=\wh S$, we have $\DDs(S)\subs\bigcap_{s\in S}\DDs(\whs)$. 
Thus, (iii) follows from (i).
\end{proof}
As to (iii), the inclusion in general does not reverse: Letting $X=\Rset$,
Borel conjecture yields $\DD(\SEQ)=\DDsh(\SEQ)$, so for $S=\SEQ$ it is consistent.
On the other hand, a Luzin set is not meager, yet it has strong measure 
zero, see Definition~\ref{smzdef}.
If $s\in\el1$, then $\DDsh(s)$ consists of meager sets only, therefore $\DD(\SEQ)\nsubseteq\DDsh(s)$. 
Since the existence of Luzin set is consistent, so is $\DD(\SEQ)\nsubseteq\DDsh(s)$.
In this case we thus have $\DD(S)\nsubseteq\DDsh(S')$ for any $S,S'\subs\SEQ$ 
as long as $\Rset\notin\DDsh(S')$.

Theorems~\ref{versus0} and~\ref{versus4} (and other results from~\cite{micro1}) have 
counterparts for sharply dominated sets. The proofs are the same, with the aid of the 
following fact that is proved in~\cite{MR3946667}:
$E\in\EE(\hm^\phi)$ if and only if there is a $\gamma$-groupable cover $\{E_n:n\in\Nset\}$ 
such that $\sum_{n\in\Nset}\phi(\diam E_n)<\infty$.
\begin{thm}\label{versus1}
Let $\phi$ be a gauge and $s\in\SEQ$. The following are equivalent.
\begin{enum}
\item $\phi\circ s\in\el1$,
\item for every metric space $X$, $\DDsh_X(\closs)\subs\EE(\hm_X^\phi)$,
\item $\DDsh_{\el2}(\closs)\subs\EE(\hm_{\el2}^\phi)$.
\end{enum}
\end{thm}
\begin{thm}\label{versus4a}
If $\phi$ is a continuous gauge, then
$\EE(\hm^\phi)=\bigcup_{\phi\comp s\in\el1}\DDsh(\closs)$.
\end{thm}

\section{Products}\label{sec:products}
In this section we look into Cartesian products of dominated sets with other small
sets. We will equip, somehow randomly, the Cartesian product of two metric spaces with the 
$\el1$-metric:
\[
  d_{X\times Y}((x_1,y_1),(x_2,y_2))=d_X(x_1,x_2)+d_Y(y_1,y_2).
\]
Let us first collect some preliminary information about the products of 
a dominated set with other sets.
\begin{exs}\label{prodex}
(i) If $S$ is multiplicatively complete, then $\DD(S)$ is a \si ideal. 
Therefore, a product $X\times Y$ of an $S$-dominated set with a countable set is
$S$-dominated.

(ii) However, for any $s\in\SEQ$ there is an $\closs$-dominated set $X$ such that
its product with \emph{any} uncountable set is not $\closs$-dominated.
This unexpected result based on Kwela sets is established in Corollary~\ref{produn}.

(iii) Under the Continuum Hypothesis there is an $\SEQ$-dominated 
(i.e., strong measure zero) set $X\subs\Rset$ such that $X\times X$ has positive
linear measure and $X+X=\Rset$.

(iv) If $S$ is countably determined, then there is a compact $S$-dominated set 
$C\subs\Cset$ such that $C+C=\Cset$.
\end{exs}

So when the first factor in the product is merely an $S$-dominated set, 
the product with another set is rarely guaranteed to be $S$-dominated.
When we strenghten domination to sharp domination, the situation is quite 
different.

\begin{defn}[Strong and sharp measure zero]\label{smzdef}
A set $E$ in a metric space is said to have
\begin{itemyze}
\item \emph{strong measure zero} (\smz) if it is $\SEQ$-dominated. 
This classical notion is due to Borel~\cite{MR1504785}. 
By the famous Galvin-Mycielsi-Solovay Theorem, $X\subs\Rset$ has \smz{} if and 
only if $X+M\neq\Rset$ for each meager set $M$, and the same remains true in any 
locally compact Polish group.
\item \emph{sharp measure zero} (\ssmz) if it is sharply $\SEQ$-dominated. 
This notion is due to Zindulka~\cite{MR3946667}. 
An analogy of Galvin-Mycielski-Solovay Theorem holds: in any locally compact
Polish group admitting an invariant metric, $X$ has \ssmz{} if an only if $X+M$ is
meager for each meager set $M$, see~\cite{MR4472524}.
\end{itemyze}
\end{defn}
Note that the \emph{Borel Conjecture} is the statement that all \smz{} sets are countable, 
and that while the Continuum Hypothesis yields a failure of Borel Conjecture, it was proved relatively
consistent by Laver~\cite{MR0422027}. The following discussion is void under Borel Conjecture.
\begin{thm}\label{SMZ}
Let $X$ be a sharply $S$-dominated metric space. 
\begin{enum}
\item If $Y$ is \smz, then $X\times Y$ is $S$-dominated.
\item If $Y$ \ssmz, then $X\times Y$ is sharply $S$-dominated.
\end{enum}
\end{thm}
\begin{proof}
Let $s\in S$ and let $\{E_n:n\in\nset\}$ a $\gamma$-groupable $s$-fine cover of $X$. 
Let $\{F_k:k\in\Nset\}$ be a partition of $\nset$ into finite sets such that the family 
$\{\bigcup_{n\in F_k}E_n:k\in\nset\}$ is a $\gamma$-cover. 
For each $k\in\nset$ let $\eps_k<\min_{n\in F_k} (s_n-\diam E_n)$ be positive. It exists since $\diam E_n<s_n$.

(i) Since $Y$ is \smz, it has an $\seq{\eps_k}$-fine $\lambda$-cover $\{S_k:k\in\nset\}$.
For $n\in\nset$ let $A_n=E_k\times S_k$ where $k$ is the unique number for which $n\in F_k$.

We claim that $\{A_n:n\in\nset\}$ is a cover of $X\times Y$.
Indeed, if $x\in X$, then $\fmany k\ x\in\bigcup_{n\in F_k}E_n$
and if $y\in Y$ then $\emany k\ y\in S_k$. Hence there is $k$ such 
that $(x,y)\in\bigcup_{n\in F_k}E_n\times S_k=\bigcup_{n\in\nset}A_n$. Since clearly 
\[
  \diam A_n=\diam E_n+\diam S_k<\diam E_n+\eps_k<s_n,
\]
it is also $s$-fine, as required, so (i) is proved.

(ii) is proved in the same manner, except that the cover $\{S_k\}$ can be chosen to be $\gamma$-groupable.
So suppose that $\{G_i:i\in\Nset\}$ be the partition of $\nset$ into finite sets such that the family 
$\{\bigcup_{k\in G_i}S_k:i\in\nset\}$ is a $\gamma$-cover. For each $i$ let $H_i=\bigcup_{k\in G_i}F_k$. 
We claim that $\{A_n:n\in\nset\}$ (where $A_n$ are defined as above) is a $\gamma$-groupable cover 
of $X\times Y$ and $H_i$'s are the witnessing families. Indeed: if $y\in Y$, 
then $\fmany i\ \exists k\in G_i \ y\in S_k$ and also if $x\in X$, 
then $\fmany k\ \exists n\in F_k\ x\in E_n$. It follows that 
\[
  \fmany  i\ (\exists k\in G_i \ y\in S_k \land \exists n\in F_k\ x\in E_n ),
\]
i.e., $\fmany i\ \exists n\in H_i\ (x,y)\in A_n$, as required.
\end{proof}
\begin{coro}\label{SMZ2}
Let $\group$ be a Polish group with a left-invariant metric and $X\subs\group$ be sharply $S$-dominated. 
\begin{enum}
\item If $Y\subs\group$ is \smz, then $X+Y$ is $S$-dominated.
\item If $Y\subs\group$ is \ssmz, then $X+Y$ is sharply $S$-dominated. \end{enum}
\end{coro}
\begin{proof}
Since the metric on $X\times Y$ is $\el1$-metric, 
the mapping $(x,y)\mapsto x+y$ is $1$-Lipschitz. The statements therefore follow 
at once from Theorem \ref{SMZ} and the obvious fact that $1$-Lipschitz maps preserve 
$S$-dominated and also sharply dominated sets.
\end{proof}
Let us point out that Theorem~\ref{SMZ} and Corollary~\ref{SMZ2} cannot be improved. 
Both ``Sharply dominated $\times$ \smz{} is sharply dominated'' and 
``Dominated $\times$ \smz{} is dominated'' consistently fail:
\begin{prop}\label{counter1}
Suppose the Continuum Hypothesis. Let $S\subs\SEQ$.
\begin{enum}
\item If $\Rset$ is not $S$-dominated, then there are a sharply $S$-dominated set 
$X$ and a \smz{} set $Y$ such that $X\times Y$ is not sharply $S$-dominated.
\item There is an $S$-dominated set $X$ and a \smz{} set $Y$ such that
$X\times Y$ is not $S$-dominated.
\end{enum}
\end{prop}
\begin{proof}
(i) Fremlin~\cite[534P]{MR3723040} proved that under CH there is a \smz{} set 
$X\subs\Rset$ such that $X+X=\Rset$. Letting $Y=\{0\}$ it is enough to notice that 
$X$ is not sharply $S$-dominated. Indeed, if it were, then $X+X$ would be, 
by Corollary~\ref{SMZ2}(i), $S$-dominated: a contradiction. 

(ii) Let $X$ be an $S$-Kwela set and $Y$ an uncountable \smz{} set. Then, by Corollary \ref{produn}, we have the assertion.
\end{proof}
\section{Microscopic and porous sets}\label{sec:porous}
We already mentioned Galvin--Mycielski--Solovay Theorem: a set $X\subs\Rset$ is
\smz{} if and only if $X+M\neq\Rset$ whenever $M\subs\Rset$ is meager.
Since \smz{} sets are dominated by $\SEQ$, we wonder if there is an analogous statement for 
the family of sets dominated by some other family of sequences.
In particular, is there any analogous theorem for microscopic sets? 
At first glance it seems that it may be the case: microscopic sets are dominated by geometric sequences; 
the corresponding notion that would parallel meager sets may be porous sets. 
And indeed, as we shall see, the sum of a microscopic set and a porous set is not big enough 
to cover the line. 
We are able to prove this fact in any Polish group equipped with a complete left-invariant metric 
(Theorem~\ref{porous1}) and establish a significantly stronger statement in the Cantor set 
and Euclidean spaces (Theorem~\ref{porous2}).

However, the converse, i.e., that a set $M\subs\Rset$ such that
$M+P\neq\Rset$ for all porous sets is microscopic spectacularly fails (Remarks~\ref{porous13}).

\begin{defn}[see e.g.~{\cite{MR2951635}}]
Let $(X,d)$ be a metric space. A set $A\subs X$ is termed
\begin{itemyze}
\item \emph{$p$-porous at a point}
  $x\in X$ if there is $\eps_x>0$
  such that for any $r\leq\eps_x$ there is $y\in X$ such that
  $B(y,pr)\subs B(x,r)\setminus A$,
\item \emph{porous} if there is $p>0$ such that $A$ is $p$-porous at each point $x\in A$,%
\item \si porous if $A$ is a countable union of porous sets.
\end{itemyze}
\end{defn}
\begin{lem}\label{technical}
Let $\group$ be a Polish group equipped with a complete left-invariant metric. 
If $P\subs\group$ is porous, then there is $q>0$ and an increasing union 
$\bigcup_{n\in\Nset}P_n\supseteq P$ such that for each $n$, $P_n$ is closed and
\begin{equation}\label{porous11}
  \forall r\leq 2^{-n}\ \forall x,y\in\group\ \exists z\in\group\ 
  B(z,qr)\subs B(x,r)\setminus(B(y,qr)+P_n).
\end{equation}
\end{lem}
\begin{proof}
Let $p>0$ be the porosity constant of $P$. For each $n$ let $P_n=\{x\in P:\eps_x\geq 2^{-n}\}$
where $\eps_x$ are the radii from the definition of porosity.
\cite[Lemma 3.5]{MR2951635} and its proof yield $p'>0$ such that for all $n$
\begin{equation}\label{porous22}
  \forall r\leq 2^{-n}\ \forall x\in\group\ \exists z\in\group\ 
  B(z,p'r)\subs B(x,r)\setminus P_n
\end{equation}
and \cite[Lemma 3.4]{MR2951635} ensures that we may suppose $P_n$'s closed.
Since the metric of $\group$ is invariant, we also have
\begin{equation}\label{porous3}
  \forall r\leq 2^{-n}\ \forall x,y\in\group\ \exists z\in\group\ 
  B(z,p'r)\subs B(x,r)\setminus(y+P_n).
\end{equation}
Moreover, by \cite[Lemma 3.4]{MR2951635}, we may suppose $P_n$'s closed.
Letting $q=\frac {p'}2$ yields~\eqref{porous11}.
\end{proof}
\begin{thm}\label{porous1}
Let $\group$ be a Polish group equipped with a complete left-invariant metric and $M,P\subs\group$.
\begin{enum}
\item If $M$ is microscopic and $P$ is porous, then $M+P\neq\mathbb{G}$,
\item If $M$ is sharply microscopic and $P$ is \si porous, then $M+P$ is meager in $\group$.
\end{enum}
\end{thm}
\begin{proof}
(i) Using recursively the above Lemma~\ref{technical} yields for each $n$
\begin{equation}\label{porous4}
  \forall x,y\in\group\ \exists z\in\group\ 
  B(z,q^{n+1})\subs B(x,q^n)\setminus(B(y,q^{n+1})+P_n).
\end{equation}
Since $M$ is microscopic, there is a $\lambda$-cover $\{B(y_n,q^{n+1}):n\in\Nset\}$ of $M$ 
by balls of radii $q^{n+1}$.
So, using recursively \eqref{porous4}, we get a sequence of balls $\seq{B(x_n,q^n)}$ satisfying
\begin{equation}\label{porous5} 
  B(x_{n+1},q^{n+1})\subs B(x_n,q^n)\setminus(B(y_n,q^{n+1})+P_n).
\end{equation}
Since the sequence $\seq{B(x_n,q^n)}$ is decreasing, there is a unique point 
$x\in\bigcap_{n\in\Nset}B(x_n,q^n)$. It follows from \eqref{porous5} that 
$x\notin\bigcup_{n\in\Nset}B(y_n,q^{n+1})+P_n$ and since $\{B(y_n,q^{n+1}):n\in\Nset\}$ 
is a $\lambda$-cover of $M$, $\bigcup_{n\in\Nset}B(y_n,q^{n+1})+P_n\supseteq M+P$.
We conclude that $x\notin M+P$, as required.

(ii) By~\cite[Lemma 3.4]{MR2951635} every porous set has a porous closure. 
We may thus suppose that $P$ is closed porous.
Proposition~\ref{sharp3}(i) yields a \si compact set $F\supseteq M$ that is $\wh\geom$-dominated.

Let $U$ be a nonempty open set. We will pick a point $x_U\in U\setminus(F+P)$.
Start with a compact ball $B(x_0,q)\subs U$, with $q$ small enough to satisfy~\eqref{porous11}.
Then proceed as in the proof of part (i), with $F$ in place of $M$. It is possible, since $F$ 
is $\wh\geom$-dominated. This construction yields the point $x_U\in U\setminus(F+P)$.
The set $\{x_U:U\subs\group\text{ open}\}$ is clearly dense and disjoint from $F+P$.
Since $F+P$ is a Lipschitz image of the product $F\times P$ of two \si compact sets,
it is also \si compact. Hence it is meager.
\end{proof}
The above proof is combinatorial. On the real line, 
one can employ fractal dimensions to get a somewhat stronger outcome. 
Let $\hdim$ and $\pdim$ denote the Hausdorff and packing dimension. 
The following lemma follows easily from the Howroyd's \cite{MR1362951} inequality 
$\hdim X\times Y\leq\hdim X+\pdim Y$ that holds for all
metric spaces and the fact that the mapping $(x,y)\mapsto x+y$ is Lipschitz for any
invariant metric in a metric group.
\begin{lem}\label{dim}
Let $M,P$ be sets in a separable metric group $\group$ equipped with an invariant metric.
If $\hdim M=0$, then $\hdim(M+P)\leq\pdim P$.
\end{lem}
This lemma immediately yields:
\begin{thm}\label{porous2}
Let $X$ be $\Cset$, $\Rset$ or $\Rset^n$. Let $P\subs X$ be \si porous.
\begin{enum}
\item If $M\subs X$ is microscopic, then $M+P$ is a null set.
\item If $M\subs X$ is sharply microscopic, then $M+P$ is an $F_\sigma$ null set.
\end{enum}
\end{thm}
\begin{proof}
Work in $\Rset^n$. Suppose that $M$ is microscopic and $P$ is porous. By the result 
of Salli \cite{MR1085645}, $\pdim P<n$. Since every microscopic set has Hausdorff dimension $0$, 
Lemma~\ref{dim} yields $\hdim(M+P)<n$ and in particular $M+P$ is a null set.
Passing to countable unions gives (i).

(ii) Suppose that $M$ is sharply microscopic and $P$ is \si porous. Proposition~\ref{sharp3}(i) 
yields a \si compact set $K\supseteq M$ that is $\wh{\geom}$-dominated. 
It is easy to check that a $\wh{\geom}$-dominated set has Hausdorff dimension $0$. 
Since every \si porous set is a union of countably many closed porous sets, 
the set $K\times P$ is $F_\sigma$ and thus so is $K+P$. 
Moreover, just as in part (i), $\hdim(K\times P)<n$, therefore also $\hdim(K+P)<n$ 
and in particular $K+P$ has Lebesgue measure zero.
\end{proof}
We now show that Theorems~\ref{porous1} and~\ref{porous2} are the best we can get.

Recall that the lower and upper density of $I\subs\Nset$ are, respectively,
\begin{equation}\label{densdef}
  \ldens I=\liminf_{n\to\infty}\frac{\abs{I\cap n}}{n},\quad
  \udens I=\limsup_{n\to\infty}\frac{\abs{I\cap n}}{n}   
\end{equation}
We need the following easy construction. See the text following Definition~\ref{def:infra}
for the unexplained notation. Let $I\subs\Nset$. Define
\[
  C_I=\{x\in\Cset:x\rest(\Nset\setminus I)=0\}.
\]
Then $C_I$ is a compact subset of $\Cset$ whose features are of course affected by the choice of $I$:
\begin{lem}\label{CI}
\begin{enum}
\item $C_I+C_{\Nset\setminus I}=\Cset$,
\item $\hdim C_I=\lpdim C_I=\lbdim C_I=\ldens I$,
\item $\pdim C_I=\ubdim C_I=\udens I$,
\item if $\Nset\setminus I$ is infinite, then $C_I$ is upper porous,
\item if $\Nset\setminus I$ is syndetic, then $C_I$ is lower porous,
\item if $I=\{2^k:k\in\Nset\}\cup\{2^k+1:k\in\Nset\}$, then
$\hdim C_I=0$, but $C_I$ is not microscopic,
\item if $S\subs\SEQ$ is countably determined, then there is a set
$I\subs\Nset$ such that $C_I\cup C_{\Nset\setminus I}$ is $S$-dominated.
\end{enum}
\end{lem}
\begin{proof}
(i) For each $x\in\Cset$ let $x_I(n)=x(n)$ if $n\in I$ and $x_I(n)=0$ otherwise.
Then clearly $x_I\in C_I$, $x_{\Nset\setminus I}\in C_{\Nset\setminus I}$ 
and $x_I+x_{\Nset\setminus I}=x$.

(ii) and (iii) are routine.

(iv) If $x\in C_I$ and $n\notin I$, then 
$\seq{x\rest n\concat 1}\cap C_I=\emptyset$. Hence the cones $\seq{x\rest n}$ have,
within $C_I$, holes of half their size on each level $n\notin I$, which is enough.

(v) By~\cite[Lemma 3.7]{MR2951635}
a set $A\subs\Cset$ is lower porous if and only if
\[
\exists n\ \forall p\in\cset\ \exists q\supseteq p\ \abs{q}=\abs{p}+n
  \wedge A\cap\seq{q}=\emptyset.
\]
Applying this to $C_I$ yields (v).

(vi) Let $\micr(r)=\frac{1}{\log(1/r)}$. This gauge was introduced in~\cite{micro1}. 
It is easy to calculate that $\frac{2^{\abs{I\cap n}}}{n^2}\geq\frac14$ for all $n$. 
Therefore, $\varliminf_{n\to\infty}\micr^2(2^{-n})2^{\abs{I\cap n}}>0$ whence
$\hm^{\micr^2}(C_I)>0$ and by~\cite[Corollary 5.8]{micro1}, $C_I$ is not microscopic. 
On the other hand, clearly $\udens(I)=0$, so by (ii) or (iii), $\hdim C_I=0$.

(vii) Since $S$ is countably determined, we may suppose that it is countable. 
For each $s\in S$ choose a gauge $\phi_s\asymp s$ and let
$\phi$ be a gauge such that $\phi\prec\phi_s$ for all $s\in S$.

Build recursively an increasing function $f\in\Pset$ subject to
\begin{equation}\label{C_I}
  2^{f(n)}\phi(2^{-f(n+1)})\leq\frac1n.
\end{equation}
Let $I=\bigcup_{n\in\Nset\text{ even}}[f(n),f(n+1))$.
Then $\abs{I\cap f(n+2)}\leq f(n+1)$ for all $n$ even, hence by~\eqref{C_I}
\begin{align*}
  \hm^\phi(C_I)&\leq\liminf_{n\to\infty}2^{\abs{I\cap f(n+2)}}\phi(2^{-f(n+2)}) \\
  &\leq\liminf_{n\text{ even}}2^{f(n+1)}\phi(2^{-f(n+2)})\leq\liminf_{n\text{ even}}\frac1{n+1}=0.
\end{align*}
On the other hand, $\Nset\setminus I=\bigcup_{n\text{ odd}}[f(n),f(n+1))$, hence
$\hm^\phi(C_{\Nset\setminus I})=0$ by the same reasoning. Letting $C=C_I\cup C_{\Nset\setminus I}$, 
we get $\hm^\phi(C)=0$. Since~\cite[Corollary 5.8]{micro1} yields $\NN(\hm^\phi)\subs\DD(\closs)\subs\DD(s)$
for all $s\in S$, we get that $\NN(\hm^\phi)\subs\bigcap_{s\in S}\DD(s)=\DD(S)$. Therefore, $C\in\DD(S)$.
\end{proof}
\begin{lem}\label{salli}
There exists a symmetric Cantor set $C\subs\Rset$ such that $\pdim C=0$ and yet $C$ is not \si porous. 
\end{lem}
\begin{proof}
We refer to~\cite[Section 3, Symmetric Cantor sets]{micro1}. The idea is simple: 
our set $C=\cube(r)$ is a symmetric Cantor set induced by a sequence $r=\seq{r_n:n\in\Nset}$ 
that converges to zero rapidly enough to ensure $\pdim C=0$, 
but we occasionally make the gaps tiny enough to destroy porosity.

Let $I=\{3,7,15,31,\dots\}=\{2^k-1:k>1\}$. 
We define $r_n$'s by recursion over elements of $I$ subject to 
the following conditions. Let $\del_n=r_n-2r_{n+1}$ be the gaps in the middle of each interval 
in the Cantor set on level $n$, i.e., of length $r_n$.
\begin{enumerate}
\item[(a)] $r_{n+1}<r_n/2$,
\item[(b)] $r_n\leq n^{-n}$,
\item[(c)] $\forall j\in I\ \del_j=\frac{r_j}{j}$,
\item[(d)] $\forall j\in I\ \forall n>j\ \del_n<\del_j$.
\end{enumerate}
So let $j\in I$ and suppose that $r_n$'s have been set for all $n<j$. Let
\begin{enumerate}
\item[(e)] $r_j=(2j)^{-2j}r_{j-1}$,
\item[(f)] $\del_j=\frac{r_j}{j}$, $r_{j+1}=\frac{r_j-\del_j}{2}$,
\item[(g)] if $j<n<2j$, then $r_{n+1}<r_n/2$,
\item[(h)] but the gaps $\del_n$ are smaller than $\del_j$.
\end{enumerate}
Note that (g) and (h) are easily achieved.

It is clear that (e), (f) and (g) yield (a).

To prove (b), suppose $j\in I$ and $n\in[j,2j+1)$. (e) ensures (b)
for $n=j$. If $n>j$, then by (f) and (g) $r_n<r_j$ and 
by (e) $r_j\leq(2j)^{-2j}\leq n^{-n}$ since $n\leq 2j$, and (b) follows.

(c) is nothing but (f).

As to (d), due to (h) it is enough to show that if $j$ and $j'=2j+1$ are two consecutive elements of $I$, 
then $\del_{j'}<\del_j$, and that follows from (e) and (f) by straightforward calculation.

Now (a) ensures that the set $C$ is indeed a symmetric Cantor set, (b) ensures small packing dimension 
and (c) with (d) destroy porosity: it is rather obvious that
\[
  \pdim C\leq\limsup_{n\to\infty}\frac{\log2^{n+1}}{\log\frac{1}{r_n}}
  \leq\limsup_{n\to\infty}\frac{n+1}{\log n^n}=0,
\]
see, e.g., \cite{MR2834788}.
On the other hand, let $x\in C$ and consider the collection of basic intervals $\{I_{x\rest j}:j\in I\}$. 
Each of these intervals $I_{x\rest j}$ has length $r_j$ and all gaps in this interval have length $\del_n$ 
for some $n\geq j$, which are by (c) and (d) at most $r_j/j$. 
This shows that no open set in $C$ is porous and by the Baire Category Theorem $C$ is not \si porous.
\end{proof}

\begin{rems}\label{porous13}
(i) Assuming CH, there is an $\SEQ$-dominated set (i.e., \smz) $X\subs\Rset$ such that $X+X=\Rset$.
$\vdash$ This is a result of Fremlin~\cite[534P]{MR3723040}, see the proof of Proposition~\ref{counter1}. $\dashv$

(ii) If $S\subs\SEQ$ is countably determined, there is a \si compact $S$-dominated set 
$X\subs\Rset$ such that $X+X=\Rset$.
$\vdash$ By Lemma~\ref{CI}(vi) and (i) there is a compact set 
$C=C_I\cup C_{\Nset\setminus I}$ that is $S$-dominated and $C+C=[0,1]$. Let $X=C+\Nset$. $\dashv$

(iii) There exist a microscopic set $M$ and a porous set $P$ in $\Rset$ such that
$M+P$ is not meager.
$\vdash$ Consider a dense $G_\del$ microscopic set which exists by 
\cite[Theorem 2.4]{micro1} and let $P=\{0\}$.  $\dashv$

(iv) There exists a compact set $M\subs\Rset$ such that $M+P$ is closed null for all \si porous
sets $P$ and yet $M$ is not microscopic.
$\vdash$ Take for $M$ a Cantor set that has Hausdorff dimension 
$0$ but is not microscopic that exists by Lemma~\ref{CI}(vi). Then use Lemma~\ref{dim}. $\dashv$

(v) Likewise, there exists a compact set $P\subs\Rset$ such that $M+P$ is (closed) null for 
all (closed) microscopic sets $M$ and yet $P$ is not \si porous.
$\vdash$ Taking for $P$ a compact set
that has packing dimension $0$ and yet it is not \si porous and using Lemma~\ref{dim} will do.
Such a set exists by Lemma~\ref{salli}. $\dashv$

(vi) There is a compact microscopic set that is not \si porous and therefore the
sum of a compact microscopic set and a compact porous set need not be \si porous.
$\vdash$ By (ii) there is a compact microscopic set $X$ such that $X+X=\Cset$. 
It cannot be \si porous, because it would contradict 
Theorem~\ref{porous2}. $\dashv$

(vii) There is a compact porous set that is not microscopic, and therefore the
sum of a compact microscopic set and a compact porous set need not be microscopic. 
$\vdash$ Let $I$ be the set of even numbers and let 
$X=C_I\cup C_{\Nset\setminus I}$. By Lemma~\ref{CI}(v), the set is porous. 
It cannot be microscopic, because it would contradict 
Theorem~\ref{porous2}. $\dashv$
\end{rems}
Note that (iv) and (v) produce a lot of examples that show that there is no hope for a 
Galvin-Mycielski-Solovay type theorem for microscopic and porous sets. 

\section{Kwela sets: core construction}\label{sec:kwela1}
As mentioned in the introduction, in~\cite{MR3482702}, Kwela proved that the additivity 
of the ideal of microscopic subsets of $\Rset$ is equal to $\omega_1$. 
In this section we extract the combinatorial core of his proof and isolate the notion of a Kwela set.

Lower density is recalled in~\eqref{densdef}.
\begin{defn}\label{def:infra}
Let $X$ be a separable metric space and $S\subseteq\SEQ$, $s\in\SEQ$.
\begin{itemyze}
\item A set $E\subs X$ is \emph{S-infradominated} if for all $s\in S$ there exists $I\subs\nset$ 
with lower density $0$ and an $s$-fine cover $\seq{E_n:n\in I}$ of $E$,
i.e., $\diam E_n<s_n$ for all $n\in I$.
\item A set $E\subs X$ is an \emph{$s$-Kwela set} if it is $\closs$-dominated, 
but not $\closs$-infradominated.
\end{itemyze}
\end{defn}
Recall that a metric $d$ is an \emph{ultrametric} if the triangle inequality reads
$d(x,z)\leq\max(d(x,y),d(y,z))$.
In the construction of Kwela sets we utilize the following simple ultrametric spaces.
Let $A$ be a set (finite or countable and often $A=2$)
and consider the set of all $A$-valued sequences $A^\Nset$.
For $x,y\in A^\Nset$ we denote by $x\wedge y$ the common initial segment of $x$ and $y$.
Thus $\abs{x\wedge y}$ is the length of this initial segment,
i.e., the first $n$ for which $x(n)\neq y(n)$.
Suppose that $r\in\SEQ$.
For $x,y\in A^\Nset$ define $d(x,y)=r_{\abs{x\wedge y}}$. It is easily verified that
$d$ is an ultrametric on $A^\Nset$.
The ultrametric space $(A^\Nset,d)$ is denoted by $\Cube(A,r)$ and termed
a \emph{Cantor cube}. In case when $r_n=\lambda^{n}$ for some constant $\lambda<1$ 
we write $\Cube(A,\lambda)$ instead of $\Cube(A,r)$.
We single out two particular cases: the space $\Cube(\Nset,\frac12)$ (usually identified
with the irrationals) is denoted just by $\Pset$ and the Cantor set $\Cube(2,\frac12)$
is denoted by $\Cset$.

Note also that for any set $A$, $m\in\nset$ and $r\in\SEQ$ the $m$-th Cartesian power
$\Cube(A,r)$ is isometric with $\Cube(A^m,r)$.

The following lemma is easy to prove.
\begin{lem}\label{cubes1}
If $A$ is a finite set and $\phi$ is a gauge, then $\hm^\phi(\Cube(A,r))=0$ if and only if
$\liminf_{n\to\infty}\abs A^n\phi(r_n)=0$.
\end{lem}

\subsection*{Construction of Kwela set}
In order to construct Kwela sets we first prove three purely combinatorial lemmas.
All trees considered are subtrees of $A^{<\omega}$ for some set $A$ 
that itself is a tree ordered by inclusion.
Notation that we will use when working with trees includes $s\para t$ for compatible nodes,
$s{\perp}t$ for incompatible nodes, $s\concat t$ for concatenation and $\cyl s$ for the cone
$\{t:t\subseteq s\}$. The family of immediate successors of a a node $t$ in the tree $T$ is 
denoted by $\suc_T(t)$.
\begin{lem}\label{spacing}
Let $p\in 4^n$ and let $A\subs\Nset\setminus(n+1)$ be such that $\ldens A>0$.
Then there exists a set $\{p_a:a \in A\}\subs 4^{<\Nset}$ satisfying the following:
\begin{enum}
\item For all $a\in A$, $\abs{p_a}=a$ and $p_a\supseteq p$,
\item $\{p_a:a \in A\}$ is an antichain in $4^{<\Nset}$,
\item For any $I\subs \Nset\setminus (n+1)$ with $\ldens(I)=0$ and any sequence 
$\{s_i:i\in I\}\subs 4^{<\Nset}$ such that $\abs{s_i}=i$ there exists $a\in A$ such that
\[
  i\in I\cap[1,a] \implies s_i\perp p_a.
\]
\end{enum}
\end{lem}
\begin{proof}
We may assume, without loss of generality, that $n=0$ and $p=\emptyset$. 
Let $A\subs\nset$ be a set with positive lower density and let $\{a_k:k\in\Nset\}$ be its increasing enumeration.

Consider now an enumeration of $3^{<\Nset}$ such that $\abs{q_i}\leq\abs{q_j}$ whenever $i<j$. 
Now, for $a_k\in A$, let $p_{a_k}\in 4^{a_k}$ be such that $q_k\concat\,3\subs p_{a_k}$. 
Clearly, the set $\{p_{a_k}:a_k\in A\}$ satisfies the first condition of the statement 
and also the second, since $p_{a_k}\subs p_{a_s}$ implies $q_k=q_s$, and therefore $k=s$. 
    
To show that it also satisfies the third one, let $I$ and $\{s_i:i\in I\}$ be as in (iii) and define, 
for each $n\in\Nset$, the sets
\begin{align*}
  L_n&=\{a_k\in A:q_k\in3^n\}, \\
  A_n&=\{a\in L_n:\exists i\in I\cap[1,n] \ p_a\para s_i\}.
\end{align*}
First we show that for $n\in\Nset$ we have $\abs{A_n}<3^n/2=\abs{L_n}/2$. 
Notice that given $n\in\Nset$ and some $i\in I$ with $i\leq n$, then for any $a_k\in L_n$ we have 
\[
  \abs{s_i}=i \leq n=\abs{q_k}<\abs{p_{a_k}}.
\]
Therefore, if $p_{a_k}\para s_i$, we must have $s_i\subs q_{k}\subs p_{a_k}$. Hence, for $i\leq n$ we get:
\begin{align*}
  \abs{\{a_k\in L_n:p_{a_k}\para s_i\}}&=\abs{\{a_k\in L_n:s_i\subs q_k\subs p_{a_k}\}}\\
  &=\abs{\{q_k\in3^n:s_i\subs q_k\}}=3^{n-i}.
\end{align*}
Thus, 
\[
  \abs{A_n}\leq \sum_{i=1}^n 3^{n-i}=\sum_{i=0}^{n-1} 3^{i}
  =\frac{3^n-1}{2}<\frac{3^n}{2}=\frac{\abs{L_n}}{2}.
\]
Using this last inequality and the fact that 
$\abs{A\cap n}=\abs{(\bigcup_{m\in\Nset}L_m)\cap n}\leq\abs{\bigcup_{m\leq n}L_n}$, 
we can show that the set $A'=\bigcup_{n\in \Nset} (L_n\setminus A_n)$ 
has positive lower density:
\begin{align*}
  \ldens A'&=\varliminf_{n\to\infty}\frac{\abs{A'\cap n}}{n}=\varliminf_{n\to\infty}
  {\frac{\abs{A\cap n}}{n}}\cdot\varliminf_{n\to\infty}\frac{\abs{A'\cap n}}{\abs{A\cap n}}\\
  &\geq\ldens A\cdot\varliminf_{n\to\infty}\frac{\tfrac{1}{2}(\abs{L_0}
  +\ldots+\abs{L_{n-1}})}{\abs{L_0}+\ldots+\abs{L_{n}}}\\
  &=\ldens A \cdot\varliminf_{n\to\infty}\frac{\tfrac{1}{2}(3^n-1)}{3^{n+1}-1} 
  =\tfrac{1}{6}\ldens A>0
\end{align*}
We claim that for each $i\in I$ there is at most one $a\in A'$ such that $s_i\para p_a$.
Assume there are $a_j,a_k\in A'$ such that $s_i\para p_{a_j}$ and $s_i\para p_{a_k}$; 
say, $a_j\in L_n\setminus A_n$ and $a_k\in L_m\setminus A_m$ with $n\leq m$. 
By definition of $A_s$, we get $i>m\geq n$, therefore there exists some $q_j\in 3^n$ 
and $q_k\in 3^m$ such that $s_i\supseteq q_j\concat3$ and $s_i\supseteq q_k\concat3$, 
so we must have $q_j=q_k$, and hence $a_j=a_k$.

To finish the proof notice that, since $\ldens(I)=0<\ldens(A')$, 
we can pick $n\in\Nset$ such that $\abs{I\cap[1,n]}<\abs{A'\cap[1,n]}$. 
Consider now the set 
\[
  B=\{a\in A'\cap[1,n]:\exists i\in I\cap[1,n]\:\:p_a\para s_i\}.
\]
Together with the last claim we get $\abs{B}\leq\abs{I\cap[1,n]}<\abs{A'\cap[1,n]}$.
Consequently, $(A'\cap[1,n])\setminus B\neq\emptyset$, which means that there exists $a\in A'\cap [1,n]$ 
such that for all $i\in I\cap[1,n]$ we have $p_a\perp s_i$, since $a<n$, 
this is true for all $i\in I\cap[1,a]$ and this finishes the proof.
\end{proof}
Recall the partition $\{\parti_k:k\in\Nset\}$, as defined in~\eqref{parti}. 
\begin{lem}\label{treeT}
There is a tree $T\subs\pset$ such that:
\begin{enum}
\item $\{t(n):t\in T, \abs{t}\geq n+1\}=\parti_n$, for all $n\in\nset$.
\item For all $t\in T$, $\ldens(\suc_T(t))>0$ and $\suc_T(t)\cap\suc_T(p)=\emptyset$ for $t\neq p$.
\end{enum}   
\end{lem} 
\begin{proof}
Let $\parti_n$ be as in the statement and notice that $\{\parti_n:n\in\Nset\}$ is a partition 
of $\nset$ into sets of positive lower density.
Let $\ell:\Nset^{<\Nset}\to \Nset$ be given by $\ell(q)=q(\abs{q}-1)$. 
    
We define the tree $T$ recursively, level by level: Let $T_0=\emptyset$, $T_1=\parti_0$ and $A_\emptyset=\parti_0$.
Assume the $n$-th level $T_n$ of $T$ has been defined. To define $T_{n+1}$ consider a partition 
of $P_n$ into sets of positive density, say $\{A_q:q\in T_n\}$, such that 
$\ldens(A_q)>0$ and $\min A_q>\ell(q)$.
Let $T_{n+1}=\{q\concat j:q\in T_n, j\in A_q\}$.
Finally define $T=\bigcup_{n\in\Nset} T_n$.
    
Observe that in this way, for $t\in T$, we have $\ldens(\suc_T(t))=\ldens(A_t)>0$ and that 
$\{\ell(t):t\in T_{n+1}\}=\parti_n$, furthermore, for all $n\in\Nset$, $\ell\restriction_{T_{n+1}}$ 
is a bijection of $T_{n+1}$ and $\parti_n$.
\end{proof}

\begin{lem}\label{mapT}
Let $T$ be as in Lemma~\ref{treeT}.
There is a monotone mapping $\varphi:T\to 4^{<\Nset}$, such that for each $a\in\Nset$ 
there exists a unique $p\in \varphi(T)=R$ such that $\abs{p}=a$.
Furthermore, for $p=\varphi(q)$, we have $\abs{p}=q(\abs{q}-1)$.  
\end{lem}
\begin{proof}
We define $\varphi:T\to 4^{<\Nset}$ recursively. 
Let $\varphi(\emptyset)=\emptyset$ and suppose that $\varphi(t)$ has been defined 
for all $t\in T_n$. Let $t\in T_{n+1}$, say $t=q\concat j$. 
Applying Lemma~\ref{spacing} to $p=\varphi(q)$ and the set $A_q=\suc_T(q)$ 
(that has positive lower density), we find a set $\{p_a:a\in A_q\}\subs 4^{<\Nset}$, 
satisfying conditions (i) and (ii) of Lemma~\ref{spacing}. 
Now we define $\varphi$ for $t=q\concat j$ as  $\varphi(t)=p_j$. 
It follows from Lemma~\ref{spacing}(i) that
$\varphi(q)\subs\varphi(t)$, and $\abs{\varphi(t)}=j=t(\abs{t}-1)$.

Let $R=\{\varphi(q):q\in T\}$. Finally, recalling that the family 
$\{A_q:q\in T\}$ is a partition of $\nset$,
by letting $p_0=\emptyset$, we can write $R=\{p_a:a\in \Nset\}$, 
where $p_a$ is the unique element in $R$ such that $\abs{p_a}=a$.
\end{proof}
We are ready for the core existence theorem.
\begin{thm}\label{KwelaSet}
For any $s\in\SEQ$, the space $\Cube(2,s)$ contains an $s$-Kwela set.
\end{thm}
\begin{proof}
We first show that $\Cube(4,s)$ contains an $s$-Kwela set.
Let $R\subs 4^{<\Nset}$ be as in Lemma~\ref{mapT}. 
Recall that $R=\{p_a:a\in \Nset\}$ with $\abs{p_a}=a$.
Define $K\subs\Cube(4,s)$ as 
\[
  K=\bigcap_{n\in\Nset}\bigcup\{\cyl{p_a}:a\in \parti_n\}.
\]
We show now that $K$ is an $s$-Kwela set. It is easy to see that Lemma~\ref{basisDomi2} 
ensures that $K$ is an $\closs$-dominated set, since, for $z_a\in\langle p_a\rangle$, 
we have $\cyl{p_a}=B(z_a, s_a)$. 
Therefore, we only need to show that it is not $\closs$-infradominated.
Take $I\subs\nset$, with $\ldens(I)=0$ and suppose that there is a sequence 
$\seq{k_i:i\in I}$ such that the cones $\seq{\cyl{k_i}:i\in I}$ form an $s$-fine cover 
of $K$. Notice that the latter yields $s_{\abs{k_i}}\leq s_i$, thus $\abs{k_i}\geq i$, 
so we may assume without loss of generality that $\abs{k_i}=i$.
We now recursively construct a strictly increasing sequence $\seq{a_n:n\in\Nset}$ such that
\begin{enum}
\item[(a)] For all $n\in\nset$, $p_{a_{n-1}} \subsetneq p_{a_n}$.
\item[(b)] For all $n\in\nset$ and for all $i\in I\cap(a_{n-1}, a_n]$, we have $k_i\perp p_{a_n}$.
\end{enum}
Let $a_0=0$ and suppose that $a_{n-1}$ has been defined. Applying Lemma~\ref{spacing} 
to $p_{a_{n-1}}\in4^{a_{n-1}}$ and the set $A_n=\suc_R(p_{a_{n-1}})\subs\Nset\setminus (a_{n-1}+1)$, 
for the set $I_n=I\setminus (a_{n-1}+1)$, we can find $a_n\in A$ such that 
$\abs{p_{a_n}}=a_n$, $p_{a_n}\supsetneq p_{a_{n-1}}$ and  
$s_i\perp p_{a_n}$, whenever $i\in I\cap(a_{n-1},a_n]$.
    
On the one hand, since $a_0=0$ and $a_n\to\infty$, (a) gives 
$\emptyset\neq\bigcap_{n\in\Nset}\cyl{p_{a_n}}\subs X$, 
on the other hand, (b) implies that for $x\in\bigcap_{n\in\Nset}\cyl{p_{a_n}}$ and $i\in\Nset$, 
we have $x\notin\cyl{ k_i}$.
Therefore, $K$ is not covered by $\{\cyl{k_i}:i\in\Nset\}$: a contradiction.

We now show that the standard bijection that maps $\Cube(2,s)$ onto $\Cube(4,s)$ preserves
Kwela sets in both directions: for $x\in\Cset$ let $\wh x\in4^\Nset$ be defined by 
$\wh x(n)=2x(2n)+x(2n+1)$ for all $n\in\Nset$. It is straightforward that if $x,y\in\Cset$ 
and $d(x,y)\leq s_{2n}$, then $d(\wh x,\wh y)\leq s_n$.
Supposing that $E\subs\Cset$ is $\closs$-dominated, for any $k\in\nset$
there is an $s^{\mult 2k}$-fine cover $\{E_n:n\in\nset\}$ of $E$. 
Hence the family $\{\wh E_n:n\in\nset\}$ is an $s^{\mult k}$-fine cover of $\wh E$. 
It is also straightforward that the inverse of $x\mapsto\wh x$ is $1$-Lipschitz, 
so if $\wh E$ is $\closs$-dominated,
then so is $E$. The invariance of infradominance is proved the same way.
It follows that since there is an $\closs$-Kwela set in $\Cube(4,s)$, there is one also in $\Cube(2,s)$.
\end{proof}
%
\section{Kwela sets: existence theorems}\label{sec:kwela2}
We now turn to the question which spaces contain Kwela sets.  
For reasons that become obvious in the next section we want a bit more.
Say that sets in a metric space are \emph{separated} 
if their lower distance is positive.
\begin{thm}\label{kwela1}
Let $X$ be $\el1$, $\el2$ or $c_0$.
For each $s\in\SEQ$ there are $\cont$ many separated $s$-Kwela sets in $X$.
\end{thm}
\begin{proof}
The Cartesian square of the Cantor cube $\Cube(2,s)$ is an ultrametric compact space. 
By~\cite{Vestfrid1994}, it isometrically embeds into $X$. 
Hence, by Theorem~\ref{KwelaSet}, $X$ contains an isometric copy 
of a product of an $s$-Kwela set and a perfect set.
\end{proof}
A sequence $s\in\SEQ$ is 
\emph{doubling} if $\exists k\ s^{\mult k}\leq\frac12s$ and 
\emph{subgeometric} if $\exists k\ s^{\shift k}\leq\frac12s$. Note that $s=\geom$ is 
subgeometric, so microscopic sets are induced by a subgeometric sequence.
\begin{thm}\label{kwela2}
For each doubling $s\in\SEQ$ there are $\cont$ many separated $s$-Kwela sets in
$\Pset$. 
\end{thm}
\begin{proof}
By~\cite{Hughes} $\Pset$ is bilipschitz universal for ultrametric compacta.
Hence, just like above, $\Pset$ contains a bilipschitz copy of a product of 
an $s$-Kwela set and a perfect set. We only have to notice that a Lipschitz map preserves 
$\closs$-dominated and $\closs$-infradominated sets. That is not true in general, 
but is easily proved for a doubling $s\in\SEQ$.
\end{proof}
\subsection*{Subgeometric sequences and nowhere porous spaces}
When looking for spaces which geometry is rich enough to admit microscopic Kwela sets, 
we formulated a notion that led us to the concept that is around for several decades. 
It was introduced by Vallin in~\cite{MR1228396} as an extreme notion of porosity. 
Unlike the other notions of porosity this one is an intrinsic property of the set, 
i.e., does not depend on the ambient space.

Let $X$ be a metric space, $x\in X$ and $0<s<r$. 
Define the \emph{shell} with center $x$ and radii $s$ and $r$ to be the set 
$S(x,s,r)=B(x,r)\setminus B(x,s)$.

We say that $X$ is \emph{strongly shell porous} at $x\in X$
if there are sequences $\seq{s_n:n\in\Nset}$ and $\seq{r_n:n\in\Nset}$ such that: 
\begin{enum}
    \item $0<s_n<r_n$ for all $n\in\Nset$,
    \item $r_n\to 0$ and $s_n/r_n\to 0$,
    \item $S(x,s_n,r_n)=\emptyset$ for all $n\in\Nset$.
\end{enum}

We introduce an auxiliary dual notion:
given $x\in X$ and $r,\eps>0$, we say that the ball $B(x,r)$ is an \emph{$\eps$-onion%
\footnote{because an onion has a lot of (nonempty) shells}%
} (or just \emph{onion}) if there is an $\eps>0$ such that $S(x,\eps s,s)\neq\emptyset$ for all $0<s\leq r$.
It is clear that $x\in X$ a center of an onion if and only if 
$X$ is not strongly shell porous at $x$.

\begin{defn}\label{oniondef}
A metric space $X$ is called an \emph{onion space} if it is complete and 
strongly shell porous at no point.
Equivalently, each of its points is the center of an onion.
\end{defn}
\begin{prop}\label{Cantonion}
For every onion space $X$ there is $\lambda>0$ such that $X$ contains a bi-Lipschitz copy 
of $\Cube(2,\lambda)$.
\end{prop}
\begin{proof}
For each $x\in X$ let $B(x,r_x)$ be the onion centered at $x$ and $\eps_x$
the corresponding constant:
\[
  \forall r<r_x\quad  S(x,\eps_x r, r_x)\neq\emptyset.
\]
For each $n\in\Nset$ let 
\[
  B_n=\{x\in X:r_x\geq\tfrac{1}{n},\eps_x\geq\tfrac{1}{n}\}.
\]
Since $X=\bigcup_{n\in\Nset}B_n$ and $X$ is a complete space, by the Baire Category Theorem
there exists $m\in\Nset$ such that $B_m$ has nonempty interior. 
Thus, for some $\eps,r\geq1/m$ there is an $\eps$-onion, say $B(x,r)$, 
such that $B(x,r)\subs \clos B_m$.
We may suppose that $\eps=r=\frac1n$. Fix this ball $B(x,r)$, 
we will need it throughout the proof.
\begin{claim}
If $B(y,s)\subs B(x,r)$, then $B(y,s)$ is an $\frac\eps4$-onion.
\end{claim}
\begin{cproof}
Write $\del=\frac\eps4$. Let $t\leq s$. Since $y\in\clos B_m$, there is $w\in B_m$ such that $d(y,w)<\del t$.
Since $w\in B_m$, the ball $B(w,\frac s2)$ is an $\eps$-onion. 
Hence, there exists $z\in B(w,\frac t2)\setminus B(w,\eps\frac t2)$, i.e., $\frac t2\eps<d(z,w)\leq\frac t2$.
Triangle inequality yields
\[
  \delta t=\tfrac t2\eps-\del t<d(y,z)<\tfrac t2+\del t<t.
\]
Hence $z\in S(y,\delta t,t)$ as required.
\end{cproof}
We now define a binary system $\seq{x_p:p\in\cset}$ within $B(x,r)$ and a system of balls
$\seq{B_p:p\in\cset}$ centered at $x_p$'s. \emph{Mutatis mutandis} we may suppose $r=1$.
Recall that $\del=\frac\eps4$ and write $r_n=(\frac\del5)^n$.

Let $x_\emptyset=x$. If $x_p$ has been defined, let $r_p=(\frac\del5)^{\abs p}$,
$B_p=B(x_p,r_p)$, $x_{p\concat0}=x_p$ and let $x_{p\concat1}$ be any point in the shell
$S(x_p,\frac12\del r_p,\frac12 r_p)$. This shell in nonempty because of the Claim. 
The balls $B_p$ have the following properties.
\begin{align}
  & p\subs q \Rightarrow B_p\supseteq B_q, \label{onion_a} \\
  & d(B_{p\concat0},B_{p\concat1})\geq \frac\del{10}\left(\frac\del5\right)^{\abs p}. \label{onion_b}
\end{align}
\begin{cproof}
\eqref{onion_a} We only need to show that $B_{p\concat 1}\subs B_p$. 
Since $d(x_p,x_{p\concat 1})\leq\frac12 r_p$, this amounts to verifying that 
$\frac12 r_p+r_{p\concat 1}\leq r_p$, which is reduces to $\frac12+\frac\del5\leq 1$, and that is trivially true.

\eqref{onion_b} The distance of balls is estimated from below by 
$d(x_p,x_{p\concat 1})-(r_{p\concat 0}+r_{p\concat 1})$.
Since $d(x_p,x_{p\concat 1})\geq\del\frac12 r_p$ and 
$r_{p\concat 0}+r_{p\concat 1}=2\frac\del5 r_p$, we can further estimate it by 
$\del\frac12 r_p-2\frac\del5 r_p=\frac\del{10}(\frac\del5)^{\abs p}$, as required.
\end{cproof}

Consider the map $\Cset\to X$ given by $\wh s=\bigcap_{n\in\Nset}B_{s\rest n}$.
Set $\lambda=\frac\del5$. Let $s,u\in\Cset$ and write $p=s\wedge u$.
Then~\eqref{onion_a} and~\eqref{onion_b} yield
\[
  \frac\del{10}\lambda^{\abs p}\leq d(B_{p\concat0},B_{p\concat1})
  \leq d(\wh s,\wh u)
  \leq\diam B_p\leq2r_p=\lambda^{\abs p}.
\]
Therefore, $s\mapsto\wh s$ is a bi-Lipschitz embedding of
$\Cube(2,\lambda)$ into $X$. 
We are done.
\end{proof}
\begin{thm}\label{tak}
If $s\in\SEQ$ is subgeometric, then in every onion space there are 
$\cont$ many separated $s$-Kwela sets.
In particular, every onion space contains $\cont$ many separated microscopic Kwela set.
\end{thm}
\begin{proof}
Let $X$ be an onion space and let $\lambda$ be the number from Theorem~\ref{Cantonion} 
for which $\Cube(2,\lambda)$ embeds into $X$.
There is $k$ such that $(\frac12)^k<\lambda^2$. Therefore 
 $u=s^{\mult k}$ satisfies $u^{\shift{}}<\lambda^2u$.
Since being $\closs$-dominated and $\clos u$-dominated are obviously equivalent, 
we may suppose that
\begin{equation}\label{embed1}
  s^{\shift{}}<\lambda^2s.    
\end{equation}
We know that $\Cube(2,s)$ contains an $\closs$-Kwela set. We will embed 
$\Cube(2,s)\xhookrightarrow{\psi}\Cube(2,\lambda^2)$ and then construct
a cascade of bilipschitz mappings
\begin{equation}\label{embed11}
  \Cube(2,s)\times\Cube(2,\lambda^2)
  \xhookrightarrow{\varphi_1}\Cube(2,\lambda^2)\times\Cube(2,\lambda^2)
  \xrightarrow{\varphi_2}\Cube(4,\lambda^2)
  \xrightarrow{\varphi_3}\Cube(2,\lambda)\xhookrightarrow{\varphi_4} X.
\end{equation}
\textbullet{} $\psi$: Let 
$f(n)=j$, where $j$ is a unique natural number such that $\lambda^{2(j+1)}\leq s_n<\lambda^{2j}$.
Due to \eqref{embed1} $f$ is a strictly increasing function.
Let $T\subs\cset$ be a tree that branches exactly on levels $f(n)$.
Let $g:\cset\to T$ be a bijection such that if $p\subs q$, then $g(p)\subs g(q)$
and the length of $g(p)$ is $f(\abs p)$ for all $p\in\cset$. 
For $x\in\Cset$ the family $\{g(x\rest n):n\in\Nset\}$ is a chain in $T$. 
Let $\psi(x)$ be the cofinal branch of $T$ determined by this chain.
Routine calculation proves that the embedding $\psi$ is bilipschitz.

\noindent\textbullet{} $\varphi_1$
is defined by the formula $\varphi_1(x,y)=(\psi(x),y)$. 
It is bilipschitz in the first coordinate and isometric in the second coordinate, so it is bilipschitz.

\noindent\textbullet{} $\varphi_2$
is defined by the formula $\varphi_2(x,y)(n)=2x(n)+y(n)$. It is an isometry.

\noindent\textbullet{} $\varphi_3$:
consider first the mapping $h:\Cset\to 4^\Nset$ defined by
$h(x)(n)=2x(2n)+x(2n+1)$. It is a bilipschitz bijection $h:\Cube(2,\lambda) \to\Cube(4,\lambda^2)$. 
Let $\varphi_3$ be the inverse of $h$.

\noindent\textbullet{} $\varphi_4$
is given by Theorem~\ref{Cantonion}, as stated at the beginning of the proof.

\noindent\textbullet{}
Let now $\varphi=\varphi_4\circ\varphi_3\circ\varphi_2\circ\varphi_1$.
It is a bilipschitz embedding of $\Cube(2,s)\times\Cube(2,\lambda^2)$ into
$X$.

We show that the image of a Kwela set under $\varphi$ is Kwela.
Lipschitz mapping does not necessarily preserve dominated or infradominated sets,
but condition~\eqref{embed1} will help us. Let $L$ be the Lipschitz constant 
of $\varphi$ and choose $j\in\nset$ large enough to satisfy 
$\lambda^{2j}L<1$. Then $s_{n+j}\leq\lambda^{2j}s_n<\frac{s_n}{L}$, and therefore if $\diam E<s_{n+j}$, then
$\diam\varphi(E)<s_n$. It follows that if $K\subs\Cube(2,s)$ is  $\closs$-dominated, then so is $\varphi(K)$,
and likewise for infradominance. Since $\varphi^{-1}$ is also Lipschitz, 
a $\varphi$-image of a set that is not $\closs$-dominated is not $\closs$-dominated,
and likewise for infradominance.
It follows that if $K$ is as $\closs$-Kwela set, then so is $\varphi(K)$.

So let $K\subs\Cube(2,s)$ be a Kwela set. Consider the family 
$\{\varphi(K\times\{x\}):x\in\Cset\}\subs X$. It consists of $\cont$ many Kwela sets that are
separated, because $\varphi^{-1}$ is Lipschitz and the sets $K\times\{x\}$ are separated.
\end{proof}
\begin{coro}
If $s\in\SEQ$ is subgeometric, then every complete space that is not totally disconnected contains 
$\cont$ many separated $\closs$-Kwela sets.
\end{coro}

\subsection*{No Kwela sets}
We now show that, maybe surprisingly, there are arbitrarily small sequences
that admit no Kwela sets in any doubling space. Recall that a metric space $X$ is \emph{doubling}
if there is $Q\in\Nset$ (the so called doubling constant) such that every ball in $X$ is
covered by at most $Q$ many balls of halved radii.
\begin{lem}\label{noKwela}
Let $X$ be a doubling metric space and $s\in\SEQ$. Assume that there exists an increasing 
function $f\in\Pset$ and $c>0$ such that $\frac{f(n-1)}{f(n)}\to0$ and 
$s_{f(n-1)}\leq cs_{f(n)-1}$ for all $n\in\Nset$. Then there is no $\closs$-Kwela set in $X$.
\end{lem}
\begin{proof}
Recall~\eqref{parti}. Let $E$ be an $\closs$-dominated set. By Lemma~\ref{basisDomi2}, 
there is an $s$-fine sequence 
$\seq{E_i:i\in\nset}$ such that $E\subs\bigcap_{k\in\Nset}\bigcup\{E_i:i\in\parti_k\}$.
Since $X$ is doubling, there is $L$ such that any set $E_i$ 
is contained in $2^L$ many sets $\{F_i^m:1\leq m\leq2^L\}$ such that $\diam F_i^m<s_i/c$. 

Fix $k\in\nset$ and let
\begin{alignat*}{2}
  &A_n&&=[f(n),f(n+1))\cap \parti_{L+k+1}, \\
  &g(n)&&=\max\{i:\abs{[i,f(n+1))\cap\parti_{k}}=2^L\abs{A_n}\} \\
  &J_n&&=[g(n),f(n+1))]\cap \parti_k,\quad J=\bigcup_{n\in\Nset}J_n.
\end{alignat*}
Consider the set $\mc F_n=\{F_i^m:i\in A_n,1\leq m\leq 2^L\}$, enumerate it by 
indices from the set $J_n$ and write the resulting sequence as $\seq{H_j:j\in J_n}$. 
For each $j\in J_n$ let $i,m$ be such that $H_j=F_i^m$. We have
\[
  \diam H_j=\diam F_i^m<\frac{s_i}{c}<\frac{cs_{f(n+1)-1}}{c}
  =s_{f(n+1)-1}\leq s_j.
\]
Therefore, $\seq{H_j:j\in J}$ is an $s$-fine cover of $E$, i.e., $\diam H_j\leq s_j$
for all $j\in J$ and since $J\subs\parti_k$, it witnesses that $E$ is 
$s^{\mult k}$-infradominated as long as we show that $\ldens J=0$.
Tedious but straightforward counting yields 
\[
  g(n)\geq\frac{f(n)+f(n+1)}{2}-2^{L+k+1}.
\]
Therefore, for all $n$ large enough,
\[
  \frac{\abs{J\cap g(n)}}{g(n)}=\frac{\abs{J\cap f(n)}}{g(n)}
  \leq\frac{f(n)}{g(n)}\leq\frac{2f(n)}{f(n)+f(n+1)-2^{L+k+2}}  
  \leq\frac{2f(n)}{f(n+1)} 
\]
and $\ldens J=\lim_{n\to0}\frac{2f(n)}{f(n+1)}=0$ by the assumption.
\end{proof}
\begin{thm}\label{noKwela2}
Let $X$ be a doubling metric space. For every $t\in\SEQ$ there is $s\in\SEQ$ 
such that $s<t$ and yet there is no $s$-Kwela set in $X$.
\end{thm}
\begin{proof}
Use Lemma~\ref{noKwela}: start with an arbitrary $f\in\Pset$ such that $f(0)=1$ and 
$\frac{f(n)}{f(n+1)}\to0$. Let $a_n=t_{f(n)}$. 
Let $s_{f(n-1)}=a_n$, choose $s_{f(n)-1}\in(\max(a_{n+1},a_n/2),s_{f(n-1)})$ and if 
$f(n-1)<i<f(n)-1$ choose $s_i$ to form a strictly decreasing sequence. 
Clearly $s\in\SEQ$ and $2s_{f(n)-1}\geq s_{f(n-1)}$. 
Apply Lemma~\ref{noKwela} to conclude that there is no $\closs$-Kwela set in $X$.
\end{proof}
%
\subsection*{Kwela sets in $\Rset$}
The above theorems let us construct Kwela sets in $\Rset$ for any subgeometric sequence 
while for a lot of other sequences there are no Kwela sets in $\Rset$.
However, we do not have any characterization of sequences for which there are 
Kwela sets in $\Rset$. In particular, we do not know if there are Kwela sets 
for generalized harmonic sequences in $\Rset$. 
Recall that $\harm^q$ is the sequence $\seq{\frac{1}{(n+1)^q}}$ and that 
$\harm^q$-dominated sets are via \eqref{a} closely related to the null sets of
the Hausdorff measure $\hm^{1/q}$. 
\begin{question}
Is there an $\harm^2$-Kwela set $K\subs\Rset$? More generally, for which $q\in(1,\infty)$ is
there an $\harm^q$-Kwela set $K\subs\Rset$?
\end{question}
\subsection*{Kwela sets and Hausdorff measures}
In~\cite{micro1} we proved that the ideals $\DD(\closs)$ and $\NN(\hm^\phi)$ are never equal in $\el2$. 
We now show another proof that utilizes Kwela sets.
The key tool is the following interesting fact.
\begin{thm}\label{uknomicro}
Let $X$ be a metric space and $\klass$ a family of separated $s$-Kwela sets in $X$.
If $\klass$ is uncountable, then $\bigcup\klass$ is not $\closs$-dominated.
\end{thm}
\begin{proof}
Suppose $\klass=\{K_\alpha:\alpha<\omega_1\}$.
Since no $K_\alpha$ is $\closs$-infradominated, for each $\alpha<\omega_1$ there is $k_\alpha\in\nset$ 
such that for each $s^{\mult k_\alpha}$-fine cover $\{U^\alpha_i:i\in I\}$ of 
$K_\alpha$ we have $\ldens(I)>0$. Clearly, there is $k\in\nset$ 
such that $\{\alpha<\omega_1:k_\alpha=k\}$ is uncountable. Therefore, we may assume 
without loss of generality that $k_\alpha=k$ for all $\alpha<\omega_1$.

Suppose that $\bigcup\klass$ is $\closs$-dominated. Then 
there is an $s^{\mult k}$-fine cover of $\bigcup\klass$, say $\{E_n:n\in\nset\}$.
For each $\alpha<\omega_1$, let $I_\alpha=\{n\in\nset:E_n\cap K_\alpha\neq\emptyset\}$.
Since $\{E_n:n\in I_\alpha\}$ is an $s^{\mult k}$-fine
cover of $K_\alpha$, it follows that $\ldens I_\alpha>0$.

We claim that $\{I_\alpha:\alpha<\omega_1\}$ is an almost disjoint family: since $d(K_\alpha,K_\beta)>0$,
there is $n\in\Nset$ such that $d(K_\alpha,K_\beta)>s_n$. Therefore, 
if $m\geq n$ then $E_m$ cannot meet both $K_\alpha$ and $K_\beta$. It follows that $I_\alpha\cap I_\beta\subs n$.
Hence $\{I_\alpha:\alpha<\omega_1\}$ is an uncountable almost disjoint family of sets with positive lower density. 
It is easy to prove that such a family does not exist, see e.g.,~\cite{MisikToth}. 
We arrived at the desired contradiction proving that $\bigcup\klass$ is not $\closs$-dominated.
\end{proof}
\begin{coro}\label{produn}
If $K$ is an $s$-Kwela set and $X$ is an uncountable metric space, then $K\times X$ is not $\closs$-dominated.
\end{coro}

\begin{thm}\label{kwelaHaus}
Let $s\in\SEQ$ and $\phi$ be a gauge.
If $X$ contains an $s$-Kwela set and $Y$ is a perfect set, 
then $\DD_{X\times Y}(\closs)\neq\NN(\hm_{X\times Y}^\phi)$.
\end{thm}
\begin{proof}
Let $K\subs X$ be an $s$-Kwela set and let $\phi$ be a gauge. In the case $\hm^\phi(K)>0$ pick arbitrary $y\in Y$ 
and let $E=K\times\{y\}$. Then clearly $E$ is $\closs$-dominated and $\hm^\phi(E)>0$, so we are done.

For the other case $\hm^\phi(K)=0$ we will employ Howroyd's  result~\cite{MR1362951} which states that
for any gauges $\psi,\tau$ and any sets $X,P$ 
\begin{equation}\label{howroyd}
  \hm^{\psi\cdot\tau}(X\times P)\leq\hm^\psi(X)\cdot\pack^\tau(P)
\end{equation}
where $\pack^\tau$ is the \emph{pseudo-packing measure}, see~\cite{MR1362951}.

Choose a gauge $\psi\prec\phi$ such that $\hm^\psi(K)=0$ and then a gauge $\tau$ 
such that $\psi\cdot\tau\geq\phi$. Then find a perfect set $P\subs Y$ such that $\pack^\tau(P)=0$.
\begin{cproof}
Build a Luzin scheme $\seq{U_p:p\in\cset}$ of closed nonempty subsets of $Y$ such that 
if $p\subs q$, then $U^p\supseteq U_q$, if $p,q$ are incompatible, then $U_p\cap U_q=\emptyset$ and
$\tau(\diam U_p)<3^{-\abs p}$. Let $P=\bigcap_{n\in\Nset}\bigcup_{p\in2^n}U_p$ be the attractor.
\end{cproof}
By~\eqref{howroyd},
\[
  \hm^\phi(K\times P)\leq\hm^{\psi\cdot\tau}(K\times P)\leq\hm^\psi(K)\cdot\pack^\tau(P)=0,
\]
while, by Corollary~\ref{produn}, $K\times P$ is not $\closs$-dominated, because $P$ is uncountable.
\end{proof}

\begin{coro}
Let $s\in\SEQ$ and $\phi$ be a gauge.
\begin{enum}
\item $\DD_X(\closs)\neq\NN(\hm_X^\phi)$ if $X$ is $\el1$, $\el2$ or $c_0$,
\item $\DD_{X}(\closs)\neq\NN(\hm_{X}^\phi)$ if $s$ is doubling and $X$ is $\Pset$,
\item $\DD_{X}(\closs)\neq\NN(\hm_{X}^\phi)$ if $s$ is subgeometric and $X$ is an onion space.
\end{enum}
\end{coro}
\begin{proof}
(i) By Theorem~\ref{KwelaSet}, there is an $s$-Kwela set $K\subs\Cube(2,s)$. By Theorem~\ref{kwelaHaus}
and its proof, there is a perfect set $P\subs\Cset$ such that $K\times P$ is not $\closs$-dominated, 
but $\hm^\phi(K\times P)=0$. Since both $K$ and $P$ are ultrametric, so is $K\times P$, and therefore it 
is isometrically embedded into $X$, witnessing (i).

(ii) By Theorem~\ref{kwela2}, there is an $s$-Kwela set $K\subs\Pset$. It is straightforward that the mapping
$F:\Pset\times\Cset\to\Pset$ defined by $F(f,x)(n)=2f(n)+x(n)$ is an isometric bijection.
Apply Theorem~\ref{kwelaHaus} to get a set $E\subs\Pset\times\Cset$ of zero Hausdorff measure that is not dominated
and consider the set $F(E)\subs\Pset$: it is of zero Hausdorff measure and it is not dominated.

(iii) 
This is proved the same way as (i) using Theorem~\ref{tak} and its proof. We refer to~\eqref{embed11}.
The only problem is to show that for a set $E\subs\Cube(2,s)\times\Cube(2,\lambda^2)$,
$\hm^\phi(E)=0$ if and only if $\hm^\phi(\varphi(E))=0$.
It follows from~\eqref{embed1} that $\Cube(2,s)$ is a doubling space, and consequently so is 
$\Cube(2,s)\times\Cube(2,\lambda^2)$. Since $\varphi$ is bilipschitz, it is thus enough to prove the 
following lemma whose routine proof is omitted.
\end{proof}
\begin{lem}
Let $Y$ is a doubling space and $\varphi:Y\to X$ a bilipschitz mapping and $E\subs Y$. 
Then $\hm^\phi(E)=0$ if and only if $\hm^\phi(\varphi(E))=0$.
\end{lem}
\section{Cardinal invariants}
\label{sec:invar}
The following are uniformity, covering, additivity and cofinality of an ideal $\II$ of subsets of $X$.
\begin{alignat*}{2}
  &\non\II &&=\min\{\abs{E}:E\notin\II\}, \\
  &\cov\II &&=\min\{\abs{\AAA}:\AAA\subs\II\land\bigcup\AAA=X\}, \\
  &\add\II &&=\min\{\abs{\AAA}: \AAA\subs\II\land\bigcup\AAA\notin\II\}, \\
  &\cof\II &&=\min\{\abs{\AAA}:\AAA\text{ is a basis of }\mathscr{I}\}. 
\end{alignat*}
We provide some information on these cardinal characteristics of the ideals of
dominated and sharply dominated sets. 

Recall that $\NN$ and $\MM$ denote the \si ideals of Lebesgue null sets and meager 
sets, respectively. The relations of the cardinal invariants of these ideals provable in ZFC are 
captured by the \emph{Cicho\'n diagram}. The symbols $\mathfrak b$ and $\mathfrak d$
denote the so called bounding and dominating numbers. 
The arrows point from the smaller to the larger cardinal.
We refer to \cite{invariants2} for details. 
\smallskip
\begin{center}
\begin{tikzpicture}[
  x=0.9cm,
  y=0.65cm,
  every node/.style={inner sep=1.5pt},
  every path/.style={->,>=stealth}
]
\node (a1) at (0,0) {$\aleph_1$};
\node (addN) at (2,0) {$\add\NN$};
\node (addM) at (4,0) {$\add\MM$};
\node (covM) at (6,0) {$\cov\MM$};
\node (nonN) at (8,0) {$\non\NN$};

\node (covN) at (2,3) {$\cov\NN$};
\node (nonM) at (4,3) {$\non\MM$};
\node (cofM) at (6,3) {$\cof\MM$};
\node (cofN) at (8,3) {$\cof\NN$};
\node (c) at (10,3) {$\cont$};

\node (b) at (4,1.5) {$\mathfrak b$};
\node (d) at (6,1.5) {$\mathfrak d$};

\draw (a1) -- (addN);
\draw (addN) -- (addM);
\draw (addM) -- (covM);
\draw (covM) -- (nonN);
\draw (nonN) -- (cofN);
\draw (cofN) -- (c);

\draw (addN) -- (covN);
\draw (covN) -- (nonM);
\draw (nonM) -- (cofM);
\draw (cofM) -- (cofN);

\draw (addM) -- (b);
\draw (b) -- (nonM);
\draw (b) -- (d);
\draw (covM) -- (d);
\draw (d) -- (cofM);
\end{tikzpicture}
\end{center}
We also make use of the \emph{transitive covering of $\MM$}
\[
  \covs\MM=\min\{\abs X:X\subs\Rset\land\exists M\in\MM\ M+X=\Rset\}.
\]
\subsection*{Additivity and cofinality}
When calculating cofinality and additivity (that are notoriously difficult to 
calculate) of $\DD(\closs)$, we will take advantage of Kwela sets, in particular their key property: 
\begin{thm}
Let $\kappa>\omega$ be a cardinal.
If $X$ contains $\kappa$ many separated $s$-Kwela sets, then $\add\DD(\closs)=\omega_1$ and 
$\kappa\leq\cof\DD(\closs)\leq\cont$. Moreover, there are $\kappa$ many disjoint subsets in $X$ 
that are not $\closs$-dominated.
\end{thm}
\begin{proof}
Let $\{K_\alpha:\alpha<\kappa\}$ be the family of separated $s$-Kwela sets. Then
by Theorem~\ref{uknomicro}, $\bigcup_{\alpha<\omega_1}K_\alpha$ is not $\closs$-dominated,
and $\add\DD(\closs)=\omega_1$ follows.

Now suppose $\mathscr B$ is a basis of $\DD(\closs)$. If $\abs{\mathscr B}<\kappa$, 
then by the pigeonhole principle there is $B\in\mathscr B$ such that 
$I=\{\alpha<\kappa:K_\alpha\subs B\}$ is uncountable and, again by Theorem~\ref{uknomicro}, 
$\bigcup_{\alpha\in I}K_\alpha\subs B$ is not $\closs$-dominated: a contradiction. Hence
$\abs{\mathscr B}\geq\kappa$. The inequality $\cof\DD(\closs)\leq\cont$ follows from  
$\DD(\closs)$ being $G_\del$-based. 

Let $\{I_\beta:\beta<\kappa\}$ be a partition of $\kappa$ into uncountable sets.
The sets $A_\beta=\bigcup_{\alpha\in I_\beta}K_\alpha$ are by Theorem~\ref{uknomicro} 
not $\closs$-dominated and  they are clearly disjoint.
\end{proof}
\begin{coro}\label{addi}
Let $s\in\SEQ$. Suppose that 
\begin{enum}
\item either $X$ is $\el1$, $\el2$ or $c_0$,
\item or $X=\Pset$ and $s$ is doubling,
\item or $X$ is an onion space and $s$ is subgeometric.
\end{enum}
Then $\add\DD_X(\closs)=\omega_1$ and $\cof\DD_X(\closs)=\cont$.
\end{coro}
On the other hand, even for a countable $S\subs\SEQ$ the additivity of $\DD(S)$ may be 
(consistently, of course) larger than $\omega_1$:
\begin{ex}
Let $\harm=\seq{\frac1{n+1}:n\in\nset}$ be the harmonic sequence 
and $X$ an analytic metric space.
The set $S=\{\harm^k:k\in\nset\}$ is countable and multiplicatively complete. 
By~\cite{micro1}, $\DD(S)=\{E\subs X:\hdim E=0\}$.
The latter is known to have additivity 
at least $\add\NN$, see, e.g., \cite{invariants2}. 
Likewise, if $q>0$ and $S$ is the multiplicative completion of $\{\harm^{1/p}:p>q\}$,
then $\DD(S)=\{E\subs X:\hdim E\leq q\}$ and hence $\add\DD(S)\geq\add\NN$.
\end{ex}

\subsection*{Kwela's problem}
Let $S\subs\SEQ$ and let $\JJ(S)$ be the family of sets $E\subs X$ such that for every
$s\in S$ there is a set $I\subs\nset$ such that $\udens I=0$ and an $s$-fine cover 
$\seq{E_n:n\in I}$ of $E$. It is clear that if $E\in\JJ(S)$, then $E$ is 
$S$-infradominated and \emph{a fortiori} $S$-dominated.

This definition of $\JJ(S)$ is inspired by Kwela~\cite[4.2]{MR3482702} who defines, 
in effect, $\JJ(\clos\geom)$ and asks in~\cite[4.7]{MR3482702} if Martin's Axiom 
implies that $\add\JJ(\clos\geom)=\cont$. We solve his problem as follows.
We will use an equivalent definition:
\begin{lem}\label{problem1}
$E\in\JJ(S)$ if and only if for each $s\in S$ there is a strictly increasing
$\varphi\in\Pset$ such that $\frac{\varphi(n)}{n}\to\infty$ and 
a $s{\circ}\varphi$-fine cover of $E$.
\end{lem}

The following theorem obviously solves Kwela's problem.
\begin{thm}\label{problem2}
Suppose that $X$ is ultrametric or \si totally bounded. If $S\subs\SEQ$, 
then $\add\JJ(S)\geq\add\NN$ and in particular $\JJ(S)$ is a \si ideal.
\end{thm}
\begin{proof}
Let $s\in S$.
Either of the assumptions yields a countable family $\{D_{n,m}:n,m\in\Nset\}$
of sets such that 
\begin{align}
  &\forall m\ \forall n\ \diam D_{n,m}<s_m, \label{problem3a} \\
  &\forall m\ \forall E\subs X\ \diam E<s_m
  \implies \exists n\ E\subs D_{n,m}.  \label{problem3b}
\end{align}
Let $\mc C=\{\varphi\in\Pset:\frac{\varphi(n)}{n}\to\infty\text{ non-decreasing}\}$.

Suppose $\{E_\alpha:\alpha<\kappa\}\subs\JJ$, $\kappa<\add\NN$.
By Lemma~\ref{problem1} and \eqref{problem3a}, there are, for each $\alpha<\kappa$,
a function $f_\alpha\in\Pset$ and $\varphi_\alpha\in\mc C$ such that 
$\seq{D_{f_\alpha(n),\varphi_\alpha(n)}:n\in\Nset}$ is a cover of $E_\alpha$.
A straightforward argument shows that since $\kappa<\add\NN\leq\mathfrak b$, there is
$\varphi\in\mc C$ such that 
\begin{equation}\label{problem4}
  \forall\alpha<\kappa\ \fmany n\ \varphi_\alpha(n)\geq\varphi(n).
\end{equation}
Let $g\in\Pset$ be nondecreasing, $g(n)\leq\sqrt[4]{\frac{\varphi(n-1)}{n}}$. 
We will utilize the Bartoszy\'nski's~\cite[2.3.9]{MR1350295} characterization of $\add\NN$ in terms of
\emph{slaloms}: since $\kappa<\add\NN$, there are functions
$F,G:\Nset\to[\Nset]^{<\omega}$ such that 
$\forall n\ \abs{F(n)}\leq g(n)\land\abs{G(n)}\leq g(n)$ and
\begin{equation}\label{problem5}
  \forall\alpha<\kappa\ \fmany n\ f_\alpha(n)\in F(n)
  \land \varphi_\alpha(n)\in G(n)\land \varphi_\alpha(n)\geq\varphi(n),
\end{equation}
the rightmost inequality by \eqref{problem4}.
For each $n$ let 
\[
  M_n=\{(i,j)\in F(n)\times G(n):j\geq\varphi(n)\}.
\]
Let $\sigma:\nset\to\bigsqcup_n M_n$ be an enumeration such that the pairs in $M_n$
precede pairs in $M_{n+1}$ and write $A_m=D_{\sigma(m)}$. We claim that the sequence 
$\seq{A_m:m\in\nset}$ is a cover of $E=\bigcup_{\alpha<\kappa}E_\alpha$ that witnesses $E\in\JJ(s)$.

Let $x\in E$. Then there is $\alpha<\kappa$ such that $x\in E_\alpha$, 
hence by \eqref{problem5} 
\[
  \fmany n\ \exists (i,j)\in M_n\ x\in D_{i,j}
\]
and it follows that $\seq{A_m:m\in\nset}$ is a cover of $E$.

Fix $n\in\nset$. Note that $\abs{M_n}\leq g^2(n)$. 
Since $\frac{\varphi(n)}{n}$ is non-decreasing, we get
$\varphi(n)-\varphi(n-1)\geq\frac{\varphi(n)}{n}>g^2(n)\geq\abs{M_n}$.
Therefore, there are more than $\abs{M_n}$ integers between $\varphi(n-1)$ and
$\varphi(n)$. It follows that there is a strictly increasing function $\tau\in\Pset$
such that $\tau(M_n)\subs(\varphi(n-1),\varphi(n)]$ for all $n$.

If $\sigma(m)\in M_n$, then, by \eqref{problem5}, $\diam A_m<s_j$ where $j\geq\varphi(n)$, 
so $\diam A_m<s_{\varphi(n)}\leq s_{\tau(m)}$. Hence $\seq{A_m:m\in\nset}$
is an $s\comp\tau$-fine cover of $E$.

It remains to show that $\frac{\tau(m)}{m}\to\infty$. If $m\in M_n$, then
$m\leq g^2(1)+\dots g^2(n)\leq n g^2(n)$ and $\tau(m)\geq\varphi(n-1)$. Hence
\[
  \frac{\tau(m)}{m}\geq\frac{\varphi(n-1)}{ng^2(n)}
  =\sqrt{\frac{\varphi(n-1)}{n}}\to\infty.
\]
Since $s\in S$ was arbitrary, we have shown that $E\in\JJ(S)$.
\end{proof}
Some remarks are in order. Following~\cite{micro1} write $\phi\asymp s$ if $\phi$ is a gauge, $s\in\SEQ$
and $\exists a,b\ \forall n$ $\frac an\leq\phi(s_n)\leq\frac bn$.
\begin{itemyze}
\item The family $\JJ(S)$ is a \si ideal even when $S$ is not multiplicatively or additively complete. 
\item Obviously $\JJ(S)\subs\DD(\clos S)$.
\item By~\cite[Theorem 5.7]{micro1}, if $\phi\asymp s$, then $\NN(\hm^\phi)\subs\JJ(s)$ 
\item and, on the other hand, if $\phi\circ s\in\el1$, then by \eqref{b} $\JJ(s)\subs\NN(\hm^\phi)$.
\end{itemyze}

\subsection*{Uniformity and covering}
If $\II\subs\JJ$ are ideals, then $\non\II\leq\non\JJ$ and $\cov\II\geq\cov\JJ$.
The ideals of Hausdorff measure zero and of dominated sets are often in inclusion:
by Theorem~\ref{versus0} and \cite[Corollary 5.8]{micro1}, if $s\in\SEQ$ and $\phi$ is a gauge, then
\begin{align}
  &\text{if $\phi\asymp s$, then $\NN(\hm^\phi)\subs\DD(\closs)$,} \label{a}\\
  &\text{if $\psi\circ s\in\el1$, then $\DD(\closs)\subs\NN(\hm^\psi)$.}\label{b}
\end{align}
Therefore, if $\phi\asymp s$, $\psi\circ s\in\el1$ and $s\in S\subs\SEQ$,
then, by \eqref{a}, Definition~\ref{smzdef} and \eqref{b}, 
\begin{equation}\label{chain}
  \smz(X)=\DD(\SEQ)\subs\DD(S)\subs\DD(\closs)\subs\NN(\hm^\psi),
  \quad\NN(\hm^\phi)\subs\DD(\closs),
\end{equation}
and consequently
\begin{align}
  &\non\smz(X)\leq\non\DD(S)\leq\non\DD(\closs)\leq\non\NN(\hm^\psi), \label{non} \\
  &\non\NN(\hm^\phi)\leq\non\DD(\closs), \label{non2} \\
  &\cov\DD(S)\geq\cov\DD(\closs)\geq\cov\NN(\hm^\psi), \label{cov} \\
  &\cov\NN(\hm^\phi)\geq\cov\DD(\closs). \label{cov2}     
\end{align}
These inequalities immediately lead to estimates of the uniformity and covering of dominated sets.
\begin{thm}\label{card1}
Let $X$ be separable and $S\subs\SEQ$ multiplicatively complete. Suppose $X$ is not $S$-dominated.
\begin{enum}
\item $\non\DD(S)\geq\cov\MM$, 
\item $\non\DD(S)\geq\covs\MM$ if $X$ is \si totally bounded,
\item $\cov\DD(S)\leq\non\MM$ if $S$ is countably determined.
\end{enum}
\end{thm}
\begin{proof}
By~\cite{invariants2}, $\non\smz(X)\geq\cov\MM$ and $\non\smz(X)\geq\covs\MM$ if $X$ is \si totally bounded. 
Apply~\eqref{non} to get (i) and (ii).

(iii) We may assume $S$ is countable. For each $s\in S$ let $\phi_s\asymp s$ 
and let $\phi$ be a gauge such that $\phi\prec\phi_s$ for all $s\in S$.
Then \eqref{a} yields $\NN(\hm^\phi)\subs\DD(S)$ and, since $X$ is not $S$-dominated, also $\hm^\phi(X)>0$.
Therefore, $\cov\DD(S)\leq\cov\NN(\hm^\phi)$. By~\cite{invariants2}, $\cov\NN(\hm^\phi)\leq\non\MM$.
\end{proof}
%
%
\begin{thm}\label{card3}
Let $X$ be $\Rset$ or $\Cset$ and $S\subs\SEQ$ multiplicatively complete.
If $X$ is not $\DD(S)$-dominated, then
\begin{enum}
\item $\covs\MM\leq\non\DD(S)\leq\non\NN$,
\item and there is $s\in\SEQ$ such that $\non\DD(s)=\covs\MM$,
\item $\cov\NN\leq\cov\DD(S)$,
\item $\cov\DD(S)\leq\non\MM$ if $S$ is countably determined,
\item but it is relatively consistent that $\cov\DD(\SEQ)>\non\MM$.
\end{enum}
\end{thm}
\begin{proof}
Since $X$ is not $\DD(S)$-dominated, there is $s$ such that $X$ is not $\DD(\closs)$-dominated.
Therefore $s\in\el1$, because otherwise $s^{\mult k}\notin\el1$ for all $k$ and thus $[0,1]$ would admit
an $s^{\mult k}$-fine cover. By \eqref{b}, $\DD(\closs)\subs\NN(\hm^1)=\NN$.
By~\cite{invariants2} $\non\NN(\hm^\phi)\leq\non\NN$ and $\cov\NN(\hm^\phi)\geq\cov\NN$ 
and (i) and (iii) follow.

(ii) Clearly $\non\DD(\SEQ)=\min_{s\in\SEQ}\non\DD(\closs)$, so (ii) follows from $\DD(\SEQ)=\smz(\Cset)$
and $\non\DD(\SEQ)=\non\smz(\Cset)=\covs\MM$, cf.~\cite[2.7]{MR1350295}.

(iv) is Theorem~\ref{card1}(iii).

(v) As shown in~\cite{Cardona2019}, $\cov\smz(\Cset)=\omega_2=\cont$ in the iterated Sacks
model, hence $\non\MM<\cov\smz(\Cset)$ holds in the model.
\end{proof}
We do not know if in ZFC there is $s\in\SEQ$ such that $\non\DD(\closs)=\non\NN$.
In $\Pset$ and Banach spaces, the situation is different.
\begin{thm}\label{card4}
If $X$ is an infinitely dimensional Banach space or $X=\Pset$
and $S\subs\SEQ$ is multiplicatively complete, then
\begin{enum}
\item $\non\DD(S)=\cov\MM$,
\item $\cov\DD(S)\geq\non\MM$, 
\item $\cov\DD(S)=\non\MM$ if $S$ is countably determined,
\item but it is relatively consistent that $\cov\DD(\SEQ)>\non\MM$.
\end{enum}
\end{thm}
\begin{proof}
Let $s\in S$ and let $\phi$ be a gauge such that $\phi\circ s\in\el1$.
By Definition~\ref{smzdef} and \eqref{b}, 
\begin{equation}\label{chain2}
  \smz(X)=\DD(\SEQ)\subs\DD(S)\subs\DD(\closs)\subs\NN(\hm^\phi).   
\end{equation}
By~\cite{invariants2}, if $X=\Pset$ or $X$ is a Polish, not locally compact group 
with an invariant metric, then $\non\NN(\hm^\phi)=\cov\MM$ and $\cov\NN(\hm^\phi)=\non\MM$.
This obviously applies when $X$ is an infinitely dimensional Banach space.
Also by~\cite{invariants2}, $\non\smz(X)=\cov\MM$, hence~\eqref{chain2} yields
\begin{align*}
  &\cov\MM=\non\DD(\SEQ)
  \leq\non\DD(S)\leq\non\DD(\closs)\leq\non\NN(\hm^\phi)=\cov\MM,  \\
  &\cov\DD(S)\geq\cov\DD(\closs)\geq\cov\NN(\hm^\phi)=\non\MM      
\end{align*}
which proves (i) and (ii). (iii) follows from (ii) and Theorem~\ref{card1}(iii).
(iv) is proved the same way as Theorem~\ref{card3}(v).
%
\end{proof}
\subsection*{Microscopic sets}
We gather information about cardinal invariants of microscopic sets, answering a few questions
posed in~\cite{MR3823475}.
\begin{thm}
\begin{enum}
\item \emph{(\cite{MR3482702})} $\add\micro(\Rset)=\omega_1$,
\item $\cof\micro(\Rset)=\cont$,
\item \emph{(\cite{MR3823475})} $\covs\MM\leq\non\micro(\Rset)\leq\non\NN$,
\item $\non\micro(\Rset)<\non\NN$ is relatively consistent,
\item $\cov\NN\leq\cov\micro(\Rset)\leq\non\MM$,
\item $\cov\NN<\cov\micro(\Rset)$ is relatively consistent,
\item $\add\MM>\cov\micro(\Rset)$ is relatively consistent.
\end{enum}
\end{thm}
\begin{proof}
(i) and (ii) follow from Corollary~\ref{addi}. 
(iii) and (v) follow from Theorems~\ref{card1} and \ref{card3}.

Let $\phi(r)=r^{1/2}$.  Since $\phi\circ\geom\in\el1$, 
\eqref{b} yields $\micro(\Rset)\subs\NN(\hm^\phi)$.

(iv) Shelah and Stepr\=ans~\cite{MR2128705} have a model of 
$\non\NN(\hm^\phi)<\non\NN$.

(vi) Elekes and Stepr\=ans~\cite{MR3910435} have a model of 
$\cov\NN=\non\NN=\mathfrak d<\cov\NN(\hm^\phi)$.

(vii) Pawlikowski~\cite{Pawlikowski1990} has a model of $\cov\smz(\Rset)<\add\MM$.
Since $\smz(\Rset)\subs\micro(\Rset)$, (vii) follows.
\end{proof}
So (v) is the best we can get in ZFC. However, it is unclear if (iii) can be improved:
\begin{question}
Is $\non\micro(\Rset)>\covs\MM$ relatively consistent?
\end{question}
Microscopic sets, have a distinctive feature --- they are preserved by H\"older mappings.
This feature lets us prove that the uniformity and covering of microscopic sets 
remains fixed across various spaces.
\begin{lem}\label{holder}
If $f:X\to Y$ is H\"older and $E\subs X$ is microscopic, then so is $f(E)$.
\end{lem}
\begin{proof}
Suppose $d(f(x),f(y))\leq C(d(x,y)^\beta$. Fix $k\in\nset$ and choose $j\geq \frac k\beta+\log_2L$.
The choice guarantees that if $\diam A<2^{jn}$, then $\diam f(A)<2^{kn}$. Therefore, the $f$-images
of a $s^{\mult j}$-fine cover of $E$ form a $s^{\mult k}$-fine cover of $f(E)$.
\end{proof}
\begin{thm}
Let $X$ be an analytic, doubling space with $\hdim X>0$.
Then 
\begin{enum}
\item $\non\micro(X)=\non\micro(\Rset)$,
\item $\cov\micro(X)=\cov\micro(\Rset)$.
\end{enum}
\end{thm}
\begin{proof}
By~\cite[Lemma 2.8]{micro1} $\non\micro([0,1])=\non\micro(\Cset)$ and $\cov\micro([0,1])=\cov\micro(\Cset)$.
We will thus work with $[0,1]$ in place of $\Cset$.
Choose $\eps\in(0,1)$ such that the snowflaked metric space $(X,d^\eps)$ satisfies $\hdim(X,d^\eps)>1$. 
Since $X$ is doubling, by the Assouad Embedding Theorem, there is $m\in\nset$ 
and a bilipschitz mapping $f:(X,d^\eps)\hookrightarrow\Rset^m$. The very same mapping
$f:(X,d)\hookrightarrow\Rset^m$ is a bih\"older embedding.
By Lemma~\ref{holder}, microscopic sets in $X$ and in $f(X)$ are the same.
Therefore, we may and will suppose that $X$ is a subset of a Euclidean space with $\hdim X>1$.

Since $X$ is analytic and $\hdim X>1$, by~\cite{KeletiMatheZindulka2014} there is a surjective 
H\"older mapping $g:X\to [0,1]$. Apply Lemma~\ref{holder} again to conclude that $\non\micro(X)\leq\non\micro([0,1])$ and $\cov\micro(X)\geq\cov\micro([0,1])$.

Since being microscopic is \si additive, we may suppose that $X$ is contained in $[0,1]^m$.
Therefore $\non\micro([0,1]^m)\leq\non\micro(X)$ and $\cov\micro([0,1]^m)\geq\cov\micro(X)$.

Consider the Peano curve $p:[0,1]\to[0,1]^m$. It is H\"older, hence 
$\non\micro([0,1])\leq\non\micro([0,1]^m)$ and $\cov\micro([0,1])\geq\cov\micro([0,1]^m)$.

The three pairs of inequalities yield the theorem.
\end{proof}
Note that unlike on $\Rset$, the two cardinal invariants are fully determined on 
$\Pset$ and Banach spaces, as immediately follows from Theorem~\ref{card4}: for any such space $X$,
$\non\micro(X)=\cov\MM$ and $\cov\micro(X)=\non\MM$.

%
\section*{Acknowledgments}
We would like to thank our colleagues Michael Hru\v s\'ak for providing 
excellent support and research milieu and Arturo Antonio Martínez-Celis Rodriguez
for providing the first insight into the failure of Galvin-Mycielski-Solovay type theorem
for microscopic and porous sets.

\bibliographystyle{amsplain}
\bibliography{micro_list,micro_extra}
\end{document}